\documentclass[letterpaper,11pt,oneside,reqno]{amsart}

\usepackage{amsfonts}
\usepackage{amsmath}
\usepackage{amssymb}
\usepackage{amsthm}
\usepackage{mathtools}
\usepackage{comment}
\usepackage[dvips]{graphicx}

\newcommand{\R}{\mathbb{R}}

\newcommand{\C}{\mathbb{C}}

\newcommand\tot{\to0 \text{ as } t\to\infty}

\newcommand{\ap}{\mathcal{AP}}
\newcommand{\aap}{\mathcal{AAP}}
\newcommand{\ml}{\mathcal}
\newcommand\apr{\ml{AP}(\R)}
\newcommand\aprc{\ml{AP}(\R,\C)}
\newcommand\aapr{\ml{AAP}(\R)}
\newcommand\aaprc{\ml{AAP}(\R,\C)}
\newcommand\aapor{\ml{AAP}_0(\R)}
\newcommand\aaporc{\ml{AAP}_0(\R,\C)}

\newcommand\lpaprc{\ml{AP}(\R,\C,p)}

\newcommand\lpaporc{\ml{AP}_0(\R,\C,p)}

\DeclareMathOperator{\sgn}{{sgn}}

\newtheorem{Theorem}{Theorem}[section]
\newtheorem{Remark}[Theorem]{Remark}
\newtheorem{Lemma}[Theorem]{Lemma}

\newtheorem{Corollary}[Theorem]{Corollary}
\newtheorem{Proposition}[Theorem]{Proposition}
\newtheorem{Definition}[Theorem]{Definition}

\numberwithin{equation}{section}

\begin{document}

	\title[Poincar\'e-Perron problem in the class of almost periodic functions]{Poincar\'e-Perron problem for high order differential equations in the class of almost periodic type functions}
	\date{\today}
	\subjclass[2020]{34A05, 34A30, 34E10, 34E05}

\keywords{Perturbed high order linear differential equation, almost periodic solutions, Riccati type equation}

\author[H. Bustos]{H. Bustos}	
	\address{H. Bustos, Centro de Docencia de Ciencias B\'asicas para Ingenier\'ia, Facultad de Ciencias de la Ingenier\'ia, Universidad Austral, Campus Miraflores, Valdivia, Chile}
	\email{harold.bustos@uach.cl}
	
\author[P. Figueroa]{P. Figueroa}	
	\address{P. Figueroa, Instituto de Ciencias F\'isicas y Matem\'aticas, Facultad de Ciencias, Universidad Austral, Campus Isla Teja, Valdivia, Chile}
\email{pablo.figueroa@uach.cl}
	
	\author[M. Pinto]{M. Pinto.}	
	\address{M. Pinto, Departamento de Matem\'aticas, Universidad de Chile, Las Palmeras 3425, Casilla 653, Santiago, Chile}
\email{pintoj.uchile@gmail.com}
	\begin{abstract}
We address the Poincar\'e-Perron's classical problem of approximation for \textsl{high order} linear differential equations in the class of almost periodic type functions, extending the results for a second order linear differential equation in \cite{FP4}. We obtain explicit formulae for solutions of these equations, for any fixed order $n\ge 3$, by studying a Riccati type equation associated with the logarithmic derivative of a solution.  Moreover, we provide sufficient conditions to ensure the existence of a fundamental system of solutions. The fixed point Banach argument allows us to find almost periodic and asymptotically almost periodic solutions to this Riccati type equation. A decomposition property of the perturbations induces a decomposition on the Riccati type equation and its solutions. In particular, by using this decomposition we obtain asymptotically almost periodic and also $p$-almost periodic solutions to the Riccati type equation. We illustrate our results for a third order linear differential equation.
	\end{abstract}
\maketitle
 
	\section{Introduction}
Several problems of differential equations, stability theory, dynamical systems and many others has been studied in connection with the development of the theory of almost periodic functions, see for instance \cite{C,Di,PRO2011}. Precisely, almost periodic functions have been useful when one looks over long enough time scales, since they come arbitrarily close to being periodic and they play an important role in describing the phenomena that are similar to the periodic oscillations which can be observed frequently in many fields such as celestial mechanics, nonlinear vibration, electromagnetic theory, plasma physics, engineering, ecosphere, life sciences and so on \cite{Anh,Cordu2009,Fe,Go,Im,JZT,PRO2010}. Furthermore, the concept of asymptotic almost periodicity was introduced in the literature \cite{Fre1,Fre2} by Fr\'echet in the early 1940s, as a natural extension of almost periodicity. Since then, this notion has generated lots of developments and applications. In particular, many authors apply the asymptotic property of asymptotically almost periodic functions to determine the existence of almost periodic solutions to various ordinary differential equations, partial differential equations, differential equations in Banach spaces, integro-differential equations as well as fractional differential equations, see for instance \cite{CaoHuang,Ch,CuHeSo,RP,Z} and the references therein. 

\medskip
On the other hand, high order differential equations have been intensively and extensively studied during the last decades. They have been attracted great attention due to their possible applications to different areas such as theoretical physics, population dynamics, biology, ecology, etc., see for example \cite{F,Om3,Ol, PW2}. In this context, the classical Poincar\'e-Perron's problem of approximation for \textsl{high order} linear differential equations consists in
\begin{equation}\label{poinca1}
	 y^{(n)}(t)+\sum_{i=0}^{n-1} (a_i+r_i(t))y^{(i)}(t)=0,
\end{equation}
where $a_0,\dots,a_{n-1}\in\mathbb{C}$ are constants and functions $r_0(t),\dots,r_{n-1}(t)$ are small in some sense with $n\ge 2$. Precisely, Poincar\'e \cite{Po} proved in 1885 the existence of a solution $y$ to \eqref{poinca1} such that $y'(t)/y(t)$ converges as $t\to +\infty$, assuming that the roots $\lambda_1,\dots,\lambda_n$ of the equation $x^n+\sum_{i=0}^{n-1} a_i x^{i}=0$ have distinct real part, $r_i$'s are continuous on $[t_0,+\infty)$ and $r_i\in C_0$, i.e., $r_i(t)\to 0$ as $t\to +\infty$ for all $i=0,\dots,n-1$. At the beginning of the twentieth century, Perron \cite{Perr} improved this result, under the same assumptions, assuring the existence of $n$ linearly independent solutions $y_1, \dots, y_n$ with the same asymptotic behaviour, namely, logarithmic derivative $y_i'(t)/y_i(t)$ converges to $\lambda_i$ as $t\to +\infty$. However, they do not provide a more precise formula for the asymptotic behaviour of solutions in either case.

\medskip
Another perturbations have been also studied, for instance $r_i\in L^p$ for some $p\ge 1$, in systems of linear differential equations, which can be applied to this equation. These results are due to Levinson \cite[Th. 1.3.1]{East} for $p=1$, to Hartman-Wintner \cite[Th. 1.5.1]{East} for $p\in (1,2]$ and to Harris-Lutz \cite{Hl} for $p>2$. Related results have been obtained by Hartman \cite[Th. 17.2]{Hartman}, \u Sim\u sa \cite{Simsa85} and Trench \cite{Trench84} assuming mild integrable conditions: $|r_i(t)|t^q$ or $r_i(t)t^q$ are integrable. In several works have been also studied cases $n=2$ in \cite{FP1}, $n=3$ in \cite{FP2,FP3} and $n=4$ in \cite{CHP1,CHP2}, recovering Poincar\'e and Perron's results and providing precise asymptotic formulae for solutions to problem \eqref{poinca1} and estimated their errors, in the class of $C_0$ and $L^p$ functions. In almost diagonal linear systems, perhaps the first results in this sense were obtained in Pinto et al. \cite{PRO2011,PR,MPST}. Recently, in \cite{BFP} we have recovered Poincar\'e and Perron's results and extended results in \cite{CHP1,CHP2,FP1,FP2,FP3} for any $n\ge 5$, namely, we have obtained a precise asymptotic formulae for solutions to problem \eqref{poinca1} and estimated their errors, in the class of $C_0$ and $L^p$ functions.

\medskip
To the best of our knowledge, concerning different type of perturbations $r_i$'s in \eqref{poinca1} such as almost periodic type functions (see Definitions \ref{fapdef} and \ref{faplpdef}) few results are available. In particular, in case $n=2$ the authors in \cite{FP4} addressed equation
$$y''+(a_1+r_1(t))y'+(a_0+r_0(t))y=0, \qquad t\in\R$$
in the class of almost periodic type functions: $r_i\in\ap(\R,\C)$ or $r_i\in\aap(\R,\C)$ or $r_i\in \ml{AP}(\R,\C,p)$, obtaining explicit formulae for solutions by studying a Riccati equation associated with the logarithmic derivative of a solution.

\medskip In this paper, inspired by results in \cite{BFP,FP4} we are interested in generalize Poincar\'e's and Perron's classical problem of approximation \eqref{poinca1} to the class of almost periodic type functions for any $n\ge 3$, namely, equation \eqref{poinca1} with either $r_i\in\ap(\R,\C)$ or $r_i\in\aap(\R,\C)$ or $r_i\in \ml{AP}(\R,\C,p)$. Precisely, we have obtained explicit formulae for solutions of these equations under sufficient conditions by studying a Riccati type equation associated with the logarithmic derivative of a solution. In particular case $n=3$, if $\lambda$, $\lambda_1$ and $\lambda_2$ are roots of $P(a;x)=x^3+a_2x^2+a_1x+a_0$ and have distinct real parts, and $r_i$, $i=0,1,2$ are sufficiently small in $L^\infty$-sense and almost periodic, then there is a solution $y:\R\to\C$ satisfying \eqref{eqn3} and
\begin{equation}\label{fth1n3}
    \begin{split}
        y_\lambda(t)=&\, e^{\lambda t}\exp\left(-\gamma_1^{-1}\gamma_2^{-1} \int_{0}^t [P(r;\lambda)(s) +\mathcal{L}(s,z(s))+\mathcal{F}(s,Z(s))]\,ds\right)\\
        &\times \exp\left(  -\sum_{j=1}^2\frac{1}{\Gamma_j\gamma_j} G_{\gamma_j}[P(r;\lambda)+\mathcal{L}(\cdot,z)+\mathcal{F}(\cdot,Z)](t)\right).
    \end{split}
\end{equation}
where
\begin{equation}\label{prn3}
        P(r(t);\lambda)=\lambda^2 r_2(t)+\lambda r_1(t) +r_0(t),
    \end{equation}
    \begin{equation}
        \mathcal{L}(t,z)=r_2(t)z'+(2\lambda r_2(t) +r_1(t))z \qquad\text{ and }
    \end{equation}
    \begin{equation}\label{fzn3}
        \mathcal{F}(t,z,z')= (a_2+r_2(t)+3\lambda)z^2  + 3 zz'+z^3
    \end{equation}
and $z$ satisfies a Riccati type equation, is an almost periodic complex value function. Some conditions (see \eqref{cl}) insure the estimate $z^{(i)}=O\big((I_{\beta}+I_{-\beta})[P(r;\lambda)]\big)$, $i=0,1$,
where $0<\beta<\min\{\Re(\lambda_j - \lambda)\ :\ j=1,2\}$ of $z$. 
We have denoted $\gamma_i=\lambda_i-\lambda$ and $\Gamma_i=(-1)^{i}(\lambda_2-\lambda_2)$ for $i=1,2$ and we have used Green's operators acting on locally integrable functions $f:\mathbb{R}\to\C$ given by
\begin{equation}\label{defls}
G_{\omega}\big[f\big](t):=\int_{-\infty}^{\infty} g_{\omega}(t,s)f(s)\,ds\quad\text{and}\quad	I_{\omega}\big[f\big](t):=\int_{-\infty}^{\infty}|g_{\omega}(t,s) f(s)|\,ds\,,
\end{equation}
where
\begin{equation}\label{defg}
g_{\omega}(t,s):=
\begin{cases}
-{\rm sgn}(\Re{\omega}) e^{\omega(t-s)} & {\rm sgn}(\Re{\omega})(t-s) < 0 \\
0 & \text{otherwise}
\end{cases} \,,
\end{equation}
for a given $\omega\in \C$ with $\Re\omega\ne 0$, and $\Re{\omega}$ denotes the real part of $\omega$. See Theorem \ref{orederzun} for the precise statement for any $n\ge 3$. Furthermore, we have obtained sufficient conditions to ensure the existence of a \textsl{fundamental system of solutions} of the form \eqref{fth1n3}, see Theorem \ref{base}.

\medskip
In general, given a $(n+1)$-tuple $c=(c_0,\dots,c_{n})$ of complex numbers,  we will denote $P(c;x):=\sum_{i=0}^{n}c_ix^{i}$, so that $P(a;x)=x^n+\sum_{i=0}^{n-1} a_i x^{i}$ is the characteristic polynomial of \eqref{poinca1} with $r_i\equiv0$, $i=0,\dots,n-1$, where $a=(a_0,\dots ,a_{n-1}, 1)$, $a_i$'s are the complex numbers associated to the equation \eqref{poinca1} and $a_n=1$. Hence, as in \cite{Perr,Po}, we will also assume that $P(a;x)$ has $n$ roots, $\lambda,\lambda_1,\dots,\lambda_{n-1}$, with different real parts, so that the characteristic polynomial of \eqref{linearz} $P_D(a,x)$ (see \eqref{defdpol}) will have roots with different and non zero real parts. Therefore an exponential dichotomy can be considered a scalar exponential dichotomy with scalar Green's functions \eqref{defg}. 
These operators $G_\omega$ and $I_\omega$ allows us to define the Green's function $\mathcal{G}$ and Green's operator $G$ for equation $\mathcal{D}z=f$ (see equations \eqref{greenfor} and \eqref{greenop}), so that an integral equation for $z$ is obtained, see \eqref{eiz}. 

 \medskip

On the other hand, these type of results (Theorems \ref{orederzun} and \ref{base}) are also true for $r_i$, $i=0,\dots,n-1$ in either $C_0$ or $L^p$, see \cite{Bellm,BFP,Co,East}. In general, $z$ is related with the logarithmic derivative of $y$ and represent the error function belonging to spaces under consideration $C_0$, $L^p$ or almost periodic type functions, and is small with respect to the norm in the space; $L^\infty$ in our case. The fixed point Banach argument allows us to find such an almost periodic or asymptotically almost periodic solution $z$ to a Riccati type equation, see \eqref{dif}. Furthermore, this procedure allows us to address $p$-almost periodic perturbations, see Definition \ref{faplpdef}. In other words, there exist
several classes of almost periodic type functions verifying these
results. This is the case of those with a summand $C_0$ or $L^p$
(asymptotically almost periodic functions or $p$-almost periodic
functions), namely, $r_i=\mu_i+\nu_i$, $i=0,\dots,n-1$ with
$\mu_i\in\aprc$ and either $\nu_i\in C_0$ or $\nu_i\in L^p$ for
$i=0,\dots,n-1$, see Theorems \ref{teoczero} and \ref{teo35n3}. This decomposition of the coefficients $r_i$, $i=0,\dots,n-1$ induces the decomposition $z=\theta+\psi$, where $\theta$ is the almost periodic part and $\psi$ is $C_0$ or $L^p$, respectively. These functions $\theta$ and $\psi$ satisfy their own equations, which can be treated and
solved separately, implying the existence of new solutions. See equations \eqref{thetaap}-\eqref{eqpsi} and Theorems \ref{teoczero}
and \ref{teo35n3}. In section 4, we will discuss these results through an illustrative
example. Finally, we point out that the used method is scalar
\cite{FP1,FP2,FP3}, reducing the order of the equation and
avoiding the usual diagonalization process \cite{Co,East,Gin1}.

	\section{Preliminaries}\label{pre}
	
	Following ideas presented in \cite{FP4}, here we present first almost periodic type functions, then Green's type operators, complete Bell polynomials and we end up this section with the Riccati equation and almost periodic functions.
	
	\subsection{Almost periodic type functions}

According to \cite{Co,LZ,Z}, the notion of almost periodic function we shall use is the following.

\begin{Definition}\label{fapdef}{\rm
\begin{itemize}
\item[$a$)] A continuous function $r:\R\to\C$ satisfies $r\in\ap(\R,\C)$ if for every sequence $\{b_k\}$ of real numbers there exists a subsequence $\{\tilde{b}_m\} $ such that $\lim_{m\to\infty}r(t+\tilde{b}_m)$ exists uniformly for $t\in\R$.

    \item[$b$)] A continuous function $F:\R\times\C^{n}\to\C$ satisfies $F\in\ap(\R\times\C^{n},\C)$ if for every sequence $\{b_k\}$ of real numbers there exists a subsequence $\{\tilde{b}_m\} $ such that $\lim_{m\to\infty}F(t+\tilde{b}_m,Z)$ exists uniformly for $t\in\R$ and for $Z$ in compact subsets of $\C^n$\,.
\end{itemize}	}
\end{Definition}

\medskip\noindent In order to perturb almost periodic functions we present the
following result, see \cite[Lemma 1]{FP4} for a proof.

\begin{Lemma}
    \label{l1}
Let $f\in\aprc$ and $p\ge 1$. If $f(t)\to c$ as $t\to+\infty$ then
$f\equiv c$. If $f\in L^p[0,+\infty)$ then $f\equiv 0$.
\end{Lemma}

\medskip
\noindent Let us introduce the following function spaces:
$BC(\R,\C)$ is the set of all bounded continuous functions from
$\R$ to $\C$. In addition, we define $BC_0(\R,\C)=\{f \in
BC(\R,\C)\mid f(t)\tot\}$, $C_{00}(\R,\C)=\{f\in BC_0(\R,\C)\mid
f(t)\to0\text{ as }t\to-\infty\},$ and $L^p_0=\{f:\R\to\C\mid
f \text{ is bounded on}$ $\text{$(-\infty,0]$ and }\int_0^\infty |f|^p<\infty\}$. In other words, from Lemma \ref{l1}
we have that $\aprc\cap BC_0(\R,\C)=\{0\}$ and $\aprc\cap
L^p_0=\{0\}$. Thus, we set

\medskip
\begin{Definition}\label{faplpdef}{\rm
\begin{enumerate}
\item[$c$)] A bounded continuous function $f:\R\to\C$ is
\emph{asymptotically almost periodic} if $f=\phi+g$ with
$\phi\in\aprc$ and $\lim_{t\to+\infty} g(t)=0$. The set of
asymptotically almost periodic functions from $\R$ to $\C$ will be
denoted by $\ml{AAP}(\R,\C)$. In addition, if $\lim_{t\to-\infty}
g(t)=0$, we will say $f\in\ml{AAP}_0(\R,\C)$. See \cite{Ch, Di}.

\item[$d$)] A bounded continuous function $f:\R\to\C$ is
\emph{p-almost periodic} with $p\ge 1$, if $f=\phi+g$ with
$\phi\in\aprc$ and $g\in L^p_0$. The set of all $p$-almost
periodic functions from $\R$ to $\C$ will be denoted by
$\ml{AP}(\R,\C,p)$. In addition, if $g\in L^p(\R)$, we will say
$f\in \ml{AP}_0(\R,\C,p)$. See \cite[section 4.3; page 46]{Di}.

\item[$e$)] A bounded continuous function $f:\R\times \C^n\to \C$
satisfied $f\in\ml{AAP}(\R\times\C^n,\C)$ if $f=\phi+g$ with
$\phi\in\ml{AP}(\R\times\C^n)$ and $\lim_{t\to+\infty} g(t,x)=0$
uniformly for $x$ on compact subsets of $\C^n$. In addition, if
$\lim_{t\to-\infty} g(t,x)=0$ uniformly for $x$ on compact subsets
of $\C$, we will say $f\in\ml{AAP}_0(\R\times \C^n,\C)$.

\item[$f$)] A bounded continuous function $f:\R\times \C^n\to \C$
satisfied $f\in \ml{AP}(\R\times\C^n,\C,p)$ if $f=\phi+g$ with
$\phi\in\ml{AP}(\R\times\C^n,\C)$ and $g(\cdot,x)\in L^p_0$
uniformly for $x$ on compact subsets of $\C^n$. In addition, if
$g(\cdot,x)\in L^p(\R)$ uniformly for $x$ on compact subsets of
$\C$, we will say $f\in \ml{AP}_0(\R\times\C^n,\C,p)$.
\end{enumerate}}
\end{Definition}

\medskip
\begin{Remark}{\rm
    It follows that
$\aprc\subset\aaporc\subset \aaprc\subset BC(\R,\C)$ and the
direct sums $\aaprc=\aprc\oplus BC_0(\R,\C)$ and
$\aaporc=\aprc\oplus C_{00}(\R,\C)$. Furthermore, $\aprc$,
$\aaporc$ and $\aaprc$ are closed subspaces of the Banach space
$BC(\R,\C)$ endowed with supremum norm $\|\cdot\|_\infty$. Also,
$\ml{AP}(\R,\C,p)=\aprc\oplus [L^p_0\cap BC(\R,\C)]$ and
$\ml{AP}_0(\R,\C,p)=\aprc\oplus [L^p(\R)\cap BC(\R,\C)]$. Notice
that $f=\phi+g$ with $\phi\in\aprc$ and $g\in L^p_0$ then
$g|_{[t_0,+\infty)}\in L^p[t_0,+\infty)$ for any $t_0\in\R$.}
\end{Remark}

\begin{Remark}{\rm
Since all the functions we will consider are from $\R$ into $\C$ for simplicity we shall omit set $\R$ and $\C$ so that we shall denote $\aprc=\mathcal{AP}$, $\aaporc=\mathcal{AAP}_0$, $BC(\R,\C)=BC$, etc.}
\end{Remark}

\subsection{Green's type operators}

It is readily checked that $I_\omega = I_{\Re \omega}$
\begin{equation}\label{cotagl}
	\big|G_\omega\big[f\big](t) \big|\leq  I_{\Re{\omega}}\big[f\big](t). 
\end{equation}
Operators $G_\omega$ and $I_\omega$ have been very useful in asymptotic integration, see for instance \cite{Bellm,CHP1,East,FP2,FP3,FP4}. They also satisfy the following inequalities, see \cite{FP2} for a proof.

\begin{Lemma}\label{techilem}
	Let  $\alpha$ be a non zero real number and  $0<\beta< |\alpha|$\,. Then 
 \begin{equation*}
	I_{\alpha}\big[aI_{\sgn(\alpha)\beta}[b]\big](t) \leq  I_{\sgn(\alpha)\beta}[b](t)I_{\alpha-\sgn(\alpha)\beta}[a](t) \,,
	\end{equation*}
$$
	I_{\alpha}[a](t)\leq I_{\sgn(\alpha)\beta}[a]\,,
	$$
 	\begin{equation*}
	I_{\alpha}\big[a\big(I_{\beta}+I_{-\beta}\big)[b]\big)\big](t)  \leq
	2\big(I_{\beta}+I_{-\beta}\big)[b](t)I_{\alpha-{\rm sgn}(\alpha)\beta}[a](t)\,.
 \end{equation*}
 \end{Lemma}
Let us recall the following result concerning the Green's
operators which has been useful in asymptotic integration, see
\cite{Bellm,Co,Cop,East,Hartman}.

\begin{Lemma}\label{oer0p}
Let $\gamma\in\C$, $\alpha=\Re\gamma\ne 0$ and let $r\in BC(\R,\C)$
be a bounded continuous function. If $r\in BC_0(\R,\C)$ (resp.
$C_{00}(\R,\C)$) then $I_\alpha[r]\in BC_0(\R,\C)$ (resp.
$C_{00}(\R,\C)$). If $r\in L^p_0$ (resp. $L^p(\R)$) for some $p\ge
1$ then $I_\alpha[r]\in BC_{0}(\R,\C)\cap L^p_0$ (resp.
$C_{00}(\R,\C)\cap L^p(\R)$). Similarly, if $r\in L^p(-\infty,0]$
for some $p\ge 1$ then $I_\alpha[r](t)\to 0$ as
$t\to-\infty$ and $I_\alpha[r]\in L^p(-\infty,0]$.
\end{Lemma}

\medskip
Notice that $G_\omega[r]$ is a solution of the linear equation
$y'=\omega y+r$. In order to study linear equations with
$r\in\aprc$ and nonlinear perturbations of them we have the
following invariant property of the Green's operator, which is straightforward from \eqref{defls}, the Definition \ref{fapdef} and Lemma \ref{oer0p}, see \cite[Lemma 5]{FP4} for a proof.

\begin{Lemma}\label{oer}
Let $\omega\in\C$, $\Re\omega\ne 0$. It holds $G_\omega:E\to E$,
where either $E=\ml{AP}(\R,\C)$, $\aaprc$ or $\aaporc$. Similarly, if
$r\in\lpaprc$ (resp. $r\in\lpaporc$) then
$G_\omega[r]\in\lpaprc\cap\aaprc$ (resp.
$G_\omega[r]\in\lpaporc\cap\aaporc$.
\end{Lemma}

\begin{Remark}{\rm
    Notice that it is also possible to consider $E=\ml{PAP}(\R,\C)$, $\ml{AA}(\R,\C)$, $\ml{AAA}(\R,\C)$ or
$\ml{PAA}(\R,\C)$, namely, the sets of pseudo-almost periodic
functions, almost automorphic functions, asymptotically almost
automorphic functions and pseudo-almost automorphic functions
respectively, see for instance \cite{CCh,Ch,CP,Di,N,P,PR,Z}.}
\end{Remark}

	\subsection{The Riccati type equation  and almost periodic functions}\label{tre}
	 In order to study \eqref{poinca1} a key step is to make the following change of variables $z=y^\prime/y-\lambda$\,, where $\lambda$ denotes a fixed root of $P(a;x)$. Then, the aim is to finding such a function $z$ depending on $r_i$'s with property $z\in\ap$. This change is equivalent to consider
	\begin{equation}\label{yeilz}
	y(t)=\exp\left(\int_{0}^t [\lambda+z(s)]\,ds\right)\,, \quad t\in\mathbb{R}.
	\end{equation}
	Hence, replacing $y$ into equation \eqref{poinca1} an equation for $z$ is obtained. Indeed, one obtains (see \cite[Section 2.3]{BFP})
\begin{equation}\label{ynbell}
\frac{y^{(i)}(t)}{y(t)}=\lambda^{i} + i \lambda^{i-1} z(t) + \sum_{j=2}^{i}{i \choose j} \lambda^{i-j}B_{j}\big(z(t),z'(t),\dots,z^{(j-1)}(t)\big)\,,
\end{equation}
where $B_j$'s are the complete Bell polynomials. Those polynomials can be defined recursively by the following formula, by setting  $B_0=1$ and for $i\in \mathbb{N}$
\begin{equation}\label{complebell}
B_{i+1}(x_1,\dots,x_{i+1})=\sum_{j=0}^{i} {i \choose j}B_{i-j}(x_1,\dots,x_{i-j})x_{j+1}\,,
\end{equation}
and satisfy the binomial type relation
\begin{equation}\label{reducedz}
	B_i(x_1+y_1,\dots,x_i+y_i)=\sum_{j=0}^{i}{i \choose j} B_{i-j}(x_1,\dots,x_{i-j})B_j(y_1,\dots,y_j)\,.
\end{equation}
The unique linear term of the complete Bell polynomial $B_i$ is $x_i$, $i\ge 1$. Moreover, define polynomials $f_i$, $i\ge 0$ by
\begin{equation}\label{bimzi}
    \begin{split}
        f_i(x_1,\dots,x_i)&=B_{i+1}(x_1,\dots,x_{i+1})-x_{i+1}=\sum_{j=0}^{i-1} {i\choose j} B_{i-j}(x_1,\dots,x_{i-j})x_{j+1},
    \end{split}
\end{equation}
in view of \eqref{complebell}. Notice that $f_0\equiv 0$. These type of polynomials and related ones have been useful in combinatorial mathematics.  They also occur in many applications, such as in the Fa\'a di Bruno's formula, and also they are related to Stirling and Bell numbers, see \cite{AndBell}. 

Summarising, we will use the following Theorem.

\begin{Theorem}{\cite[Theorem 1.1]{BFP}} \label{theo1}
Function $y$ given by \eqref{yeilz} is a solution to \eqref{poinca1} in $\mathbb{R}$ if and only if $z$ is a solution to
\begin{equation}\label{dif}
\mathcal{ D}z(t)+P(r(t);\lambda)+\mathcal{L}(t,z)+\mathcal{ F}(t,z,\dots,z^{(n-2)})=0\,, \ \ \text{in $\mathbb{R}$,}
\end{equation}
where $\lambda$ is a fixed root of $P(a;x)$, the differential linear part with constant coefficients is the $n-1$ order differential operator
\begin{equation}\label{linearz}
	\mathcal{D}z = \sum_{j=1}^{n} \frac{1}{j!}\frac{\partial^{j}}{\partial x^j}P(a;\lambda)z^{(j-1)},
\end{equation}
the function depending only on perturbations $r=(r_0,\dots,r_n)$ with $r_n\equiv 0$ is
\begin{equation}\label{prl}
P(r(t);\lambda)=\sum_{k=0}^{n-1} \lambda^k r_k(t)=r_0(t)+\lambda r_1(t) + \lambda^2 r_2(t)+\cdots + \lambda^{n-1} r_{n-1}(t),
\end{equation}
the differential linear part with variable coefficients is
\begin{equation}\label{onedee}
\mathcal{L}(t,z):=\sum_{k=1}^{n-1} \frac{1}{k!}\frac{\partial^{k}}{\partial x^k}[P(r;\lambda)]z^{(k-1)}(t).
\end{equation}
and the nonlinear term is given by
\begin{equation}\label{formf}
	\mathcal{F}(t,z,z',\dots,z^{(n-2)})=\sum_{i=2}^{n}\sum_{j=0}^{n-i}{i+j \choose j}(a_{i+j}+r_{i+j}(t))\lambda^j
    f_{i-1}(z,z',\dots,z^{(i-1)})\,.
\end{equation}
\end{Theorem}

\noindent Note that, the previous formulae above hold for $\lambda=0$ by assuming $\lambda^0=1$\,. 

In order to study Eq. \eqref{dif}, it will be useful to understand operator $\mathcal{D}$ so that we denote $P_\mathcal{D}(a;x)$ the characteristic polynomial associated to the linear differential operator $\mathcal{D}$, i.e.,
	\begin{equation}\label{defdpol}
		P_\mathcal{D}(a;x):=\sum_{k=1}^{n} \frac{1}{k!}\frac{\partial^{k}}{\partial x^k}P(a;\lambda)x^{k-1}\,
	\end{equation}
and we will use following result of \cite{BFP}.

\begin{Lemma}\label{dpol}
		Suppose that  $\lambda_l$\,, for $l=1,\dots n-1$ are the others different roots of the polynomial $P(a;x)$. Then the roots of  $P_{\mathcal{D}}(a;x)$\, are  $\lambda_l-\lambda$\,,	for $l=1,\dots n-1$\,. Furthermore, if $P(a,x)$ has roots with different real parts, then the polynomial $P_\mathcal{ D}(a,x)$ has roots with different and non zero real parts.
\end{Lemma}

\subsection{A more general non-linear equation}\label{apfre}

As in \cite{BFP}, we will study, in general way, equations of the form
\begin{equation}\label{generaleq}
	\mathcal{ D}z(t)=b(t)+R(t,z(t),\dots,z^{(n-2)}(t))\,\quad t\in\mathbb{R},
\end{equation}
where $\mathcal{D}$ is a differential operator of degree $n-1$ with constant coefficient and, the characteristic polynomial, associated to the differential operator $\mathcal{ D}$, has $n-1$ roots with real part not zero. Precisely, 	let us denote by $\{\gamma_m\}_{m=1}^{n-1}$ the  $n-1$ different roots of the characteristic polynomial of $\mathcal{ D}$\,. We also denote $\alpha_m:=\Re(\gamma_m)$\,, where $\Re$ denotes the real part. Also, we shall assume that $R:\R\times \C^{n-1}:\to\C$. We will impose some restrictions on $R$ later. 

\begin{Remark}
	{\rm 
	Note that the reduction of the Eq \eqref{poinca1} was rewritten in Eq. \eqref{dif}. This equation is of type Eq. \eqref{generaleq} denoting $b(t)=-P(r(t);\lambda)$ and $R(t,z,\dots,z^{(n-2)})=-\mathcal{L}(t,z(t))-\mathcal{ F}(t,z,\dots,z^{(n-2)})\,.$
}
\end{Remark}

We will denote by $Z:=(z_0,\dots,z_{n-2})$ for a $(n-1)$-tuple, where $z_i\in\ap(\R)$ for $i=0,1,\dots,n-2$\,. Arguing as in \cite[Lemma 7]{FP4} we have the following fact. We ommit its proof.

\begin{Lemma}\label{rtzap}
	If $R\in \ap(\R\times\C^{n-1})$ (resp. $\mathcal{AAP}(\R\times\C^{n-1})$ or $\mathcal{AAP}_0(\R\times\C^{n-1})$) and $z_0,z_1,\dots,z_{n-2}\in\ap(\R)$ (resp. $\mathcal{AAP}(\R)$ or (resp. $\mathcal{AAP}_0(\R)$)), then the map $t\mapsto R(t,z_0(t),\dots,z_{n-2}(t))\in\ap(\R)$ (resp. $\mathcal{AAP}(\R)$ or (resp. $\mathcal{AAP}_0(\R)$)).
\end{Lemma}

\medskip

The method of the variation of constants lead us to define the Green function of the homogeneous equation $\mathcal{D}z=0$. Thus, we define the Green function $\ml G(t,s)$ by 
\begin{equation}\label{greenfor}
\ml G(t,s)= \sum_{i=1}^{n-1}\frac{1}{\Gamma_j}g_{\gamma_i}(t,s)\quad \text{with}\quad \Gamma_{i}:= \prod_{k= 1,\,k\not=i}^{n-1} (\gamma_i-\gamma_k),
\end{equation}
where $g_\gamma$ is given by \eqref{defg}. Therefore, our Green operator for equation $\mathcal{D}z=f$ is given by
\begin{equation}\label{greenop}
G\big[f\big](t):=\int_{-\infty}^{\infty} \ml G(s,t)f(s)\,ds = \sum_{j=1}^{n-1}\frac{1}{\Gamma_j}G_{\gamma_j}\big[f\big](t)\,.    
\end{equation}
Notice that for any $i=1,\dots,n-2$ we have that
\begin{equation}\label{derivategi}
	G\big[f\big]^{(i)}(t)= \sum_{j=1}^{n-1}\frac{\gamma_j^i}{\Gamma_j}	G_{\gamma_j}\big[f\big](t)
\end{equation}
and
\begin{equation}\label{derivategnm1}
	G\big[f\big]^{(n-1)}(t)= \sum_{j=1}^{n-1}\frac{\gamma_j^{n-1}}{\Gamma_j}	G_{\gamma_j}\big[f\big](t) +f(t),
\end{equation}
so that $\mathcal{D}(G[f])=f,$ since
$$\sum_{j=1}^{n-1}\frac{\gamma_j^i}{\Gamma_j}=\begin{cases}
0& i=0,\dots,n-2\\
1&i=n-1
\end{cases}.
$$
Hence, a direct consequence of Lemma \ref{oer}, \eqref{derivategi} and \eqref{derivategnm1} is the following result.

\begin{Lemma}\label{greenapn}
Suppose that   $f\in\ap(\R)$\,, then the map $G[f]^{(i)}\in\ap(\R)$ for $i=0,\dots,n-1$\,.
\end{Lemma}

\begin{Remark}\label{cotadervG}
	{\rm 
		Note that from \eqref{derivategi} one can to estimate for a function $f\in\ap(\R)$ the value of $|G\big[f\big]^{(i)}(t)|$ in the following way
		\begin{equation*}
			\begin{split}
			|G\big[f\big]^{(i)}(t)| &\leq \sum_{j=1}^{n-1}\left|{\gamma_j^{i}\over \Gamma_j}G_{\gamma_j}[f](t) \right| 
\leq  \sum_{j=1}^{n-1} \frac{|\gamma_j|^{i}}{|\Gamma_j|} I_{\Re\gamma_j}[f](t) \leq \|f\|_\infty \sum_{j=1}^{n-1} \frac{|\gamma_j|^{i}}{|\Gamma_j\Re\gamma_j|} \,,
			\end{split}
		\end{equation*}
  in view of $I_{\Re\gamma_j}[1](t)=\frac{1}{|\Re\gamma_j|}$ for all $t\in\R$.
	}
\end{Remark}

\medskip

Given $Z\in \C$, with $Z=(z_0,z_1,\dots,z_{n-2})$\, we consider the norm $\|Z\|=\sum_{i=0}^{n-2}|z_{i}|$\,. Returning to the study of Eq. \eqref{generaleq}, we shall assume that there exists a constant $M>0$ and a  function $\xi_M:\R\to[0,\infty)$ such that for all $t\in\R$ and $Z_i\in\C^{n-1}$ with $\|Z\|\leq M$\,, for $i=1,2$
\begin{equation}\label{lipchitz}
    |R(t,Z_1)-R(t,Z_2)|\leq \xi_M(t)\|Z_1-Z_2\|\,.
\end{equation}

Now, given a function $z:\R \to \mathbb{C}$, for simplicity we shall denote $Z=Z(t)$ the $(n-1)$-tuple $Z=(z,z',\dots,z^{(n-2)})$. Let us introduce $\ap^{n-2}(\R)$\,, as the space of functions with $n-2$ derivatives in $\ap(\R)$, that is, the function space $\ap^{n-2}(\R)$ is defined by all functions $z:\R\to \mathbb{C}$ such that 
$z^{(i)}\in\apr$, $i=0,1,\dots,n-2$. With a slight abuse of notation we denote $\|Z(t)\|=\sum_{i=0}^{n-2}|z^{(i)}(t)|$\,. Observe that $\ap^{n-2}(\R)$ is a Banach space with norm $\|z\|_{\mathcal{AP}^{n-2}}=\sup_{t\in\R}\|Z(t)\|$.

For the next result, we denote $\tilde\alpha_j=\sum_{i=0}^{n-2}|\gamma_j^{i}|>0$ for $j=1,\dots,n-1$ and $\tilde\gamma=\sum_{k=1}^{n-1}\frac{\tilde\alpha}{|\Gamma_k|}$. Recall
that $\alpha_j=\Re\gamma_j$.

\begin{Proposition}\label{fixord}
Suppose that $b\in\apr$, and $R\in \mathcal{AP}(\R\times\C^{n-1})$. Moreover, assume that $R(\cdot,0)=0$ on $\R$, and $R$ satisfies \eqref{lipchitz} for some constant $M>0$. If there exists a positive constant $0<\epsilon_0< 1$ such that 
\begin{equation}\label{intgcond0}
	\Big\| \sum_{j=1}^{n-1}\frac{\tilde\alpha_j}{\left|\Gamma_j\right|} I_{\alpha_j}\big[\xi_M\big] \Big\|_\infty \leq \epsilon_0 \,, \quad\ \, 
	\left\| \sum_{i=0}^{n-2}\left|G[b]^{(i)}\right|\right\|_\infty \leq(1-\epsilon_0)M\,,
\end{equation}
where $\Gamma_j$ is defined in Eq. \eqref{greenfor}, then equation \eqref{generaleq} has unique solution $z$ on $\R$ such that $z^{(i)}\in\apr$ for all $i=0,1,\dots,n-2$ and satisfies the integral equation
\begin{equation}\label{eizg}
    z=G\big[b+ R(\cdot,z,z',\dots, z^{(n-2)})\big].
\end{equation}
If, in addition, there exists $0<\beta<\min\{|\alpha_j| \mid 1\leq j\leq n-1  \}$ such that 
\begin{equation}\label{intgcond}
	\Big\| \sum_{j=1}^{n-1}\frac{\tilde\alpha_j}{\left|\Gamma_j\right|} I_{\alpha_j-{\rm sgn}(\alpha_j)\beta}\big[\xi_M\big] \Big\|_\infty < \frac{1}{2},
\end{equation}
then 
\begin{equation}\label{cotaz}
    z^{(i)}=O\big( (I_{\beta}+I_{-\beta})[b] \big)\qquad \text{for all $i=0,1,\dots,n-2$}\,.
\end{equation}
\end{Proposition}

\begin{proof} 
First, for functions $z\in \ap^{n-2}(\R)$ define the operator 
	$$
	Tz(t)= \int_{-\infty}^{+\infty} \mathcal{G}(t,s)\big[b(s)+R(s, Z(s))\big]\,ds =G\big[b+R(\cdot,Z)\big](t)\,,
	$$
	where $Z=(z,z',\dots,z^{(n-2)})$. Notice that from Lemmata \ref{greenapn} it follows that $G[b]^{(i)}\in\ap$ for $i=0,\dots,n-2$. Furthermore, from the definition of Green's operator $G$ and its derivatives \eqref{derivategi} and Lemmata \ref{rtzap} and \ref{greenapn} it follows that $G[R(\cdot,Z)]^{(i)}\in\ap$ for $i=0,\dots,n-2$, since $R(\cdot,Z)\in \ap$. Thus, we conclude that $(Tz)^{(i)} \in\ap$ for $i=0,\dots,n-2$. Set
	$B_M=\big\{z\in\ap^{n-2} \, :\, \|z\|_{\mathcal{AP}^{n-2}} \leq M \,\big\}$.
	Note that $B_M$ is a closed subset of $\ap^{(n-2)}$. Also, if $z\in B_M$\,, we have that $\|z^{(i)}\|_\infty\leq M$ for all $i=0,\dots,n-2$. Hence, arguing as in \cite[Proposition 2.9]{BFP} it follows that $T(B_M)\subseteq B_M$ and $T$ is a contractive operator. Therefore, there exists a unique fixed point $z\in B_M$\,. 

\medskip
Now, consider the sequence $\{z_k\}_k\subseteq \mathcal{AP}^{n-2}$ given by $z_0=0$ and for $k\ge 0$ $z_{k+1}(t)=Tz_k(t)$. Again arguing as in \cite[Proposition 2.9]{BFP} it follows estimate \eqref{cotaz}, by using Lemma \ref{techilem}, \eqref{intgcond} and that $z_k\to z$ as $k\to +\infty$ in $\mathcal{AP}^{n-2}$, since $T$ is a contractive operator. This concludes the proof.
\end{proof}

\subsection{Other spaces of solutions} 
Considering the notation of the Proposition \ref{fixord}, let $V$ be a closed subspace of 
$$BC^{n-2}(\R,\C)=\{f:\mathbb{R}\to \mathbb{C}\mid f^{(i)}\in BC(\R,\C),\ i=0,1,\dots,n-2\},$$
$$ \text{with norm}\quad \|f\|_{BC^{n-2}}=\sup_{t\in\R} \sum_{i=0}^{n-2}\|f^{(i)}(t)\|.$$
which is invariant under the operator $T$. Then a similar result to Proposition \ref{fixord} hold if $b\in V$, under assumptions \eqref{intgcond0} and \eqref{intgcond}. Thus, there is a solution $z$ of \eqref{generaleq} and \eqref{eizg}
such that $z\in B\cap V$, where $B=\{g\in BC^{n-2}\mid \|g\|_{BC^{n-2}}\le M\}$. For instance, $\mathcal{AAP}^{n-2}=\mathcal{AP}^{n-2}\oplus BC_0^{n-2}$ is a closed subspace of $BC^{n-2}(\R, C)$ invariant under $T$, where
$$\mathcal{AAP}^{n-2}=\mathcal{AP}^{n-2}\oplus BC_0^{n-2}=\{f:\mathbb{R}\to \mathbb{C}\mid f^{(i)}\in \mathcal{AAP}(\R,\C),\ i=0,1,\dots,n-2\}.$$

However, notice that
$$W_1=\{f:\mathbb{R}\to \mathbb{C}\mid f^{(i)}\in L_0^p\cap BC(\R,\C),\ i=0,1,\dots,n-2\}\qquad\text{and}$$
$$W_2=\{f:\mathbb{R}\to \mathbb{C}\mid f^{(i)}\in L^p(\R)\cap BC(\R,\C),\ i=0,1,\dots,n-2\},$$
are invariant under $T$, but they are not closed subspaces of $BC^{n-2}(\R,\C)$, so that, we
cannot obtain directly a version of Proposition \ref{fixord} in the
subspaces either $\mathcal{AP}^{n-2}(\R,\C,p)$ or $\mathcal{AP}^{n-2}_0(\R,\C,p)$, where
$$\mathcal{AP}^{n-2}(\R,\C,p)=\{f:\mathbb{R}\to \mathbb{C}\mid f^{(i)}\in \mathcal{AP}(\R,\C,p),\ i=0,1,\dots,n-2\}\qquad\text{and}$$ 
$$\mathcal{AP}^{n-2}_0(\R,\C,p)=\{f:\mathbb{R}\to \mathbb{C}\mid f^{(i)}\in \mathcal{AP}_0(\R,\C,p),\ i=0,1,\dots,n-2\}.$$ 
Despite of this loss of
completeness, we shall find such solutions in $\mathcal{AP}^{n-2}(\R,\C,p)$ or \linebreak $\mathcal{AP}^{n-2}_0(\R,\C,p)$ by exploiting a decomposition property.

\medskip

\noindent Arguing as in \cite[Corollary 1]{FP4}, we have that the
decomposition in a sum of $b$ and $R$ induces the direct sum of
the solution $z$ and the equation \eqref{eiz} in the direct sum \eqref{tpf}, in view of
$\mathcal{AAP}^{n-2}=\mathcal{AP}^{n-2}\oplus BC_0^{n-2}$. Thus, we conclude the following

\begin{Corollary}\label{o2}
Suppose that $b=\phi+g\in\aaprc$ with $\phi\in\aprc$, $g\in BC_0(\R,\C)$
and $R=\varphi+h\in\ml{AAP}(\R\times\C^{n-1},\C)$ with $\varphi\in\ml{AP}(\R\times\C^{n-1},\C)$ and $h(\cdot,x)\in BC_0(\R,\C)$
uniformly for $x$ in compact subsets of $\C^{n-1}$. If $R$ satisfies \eqref{lipchitz} and \eqref{intgcond0} for some $M>0$, then
there is a solution $z$ of \eqref{generaleq} such that $z^{(i)}\in\aaprc$
(resp. $z^{(i)}\in\aaporc$), $i=0,\dots,n-2$, is a solution of the integral equation
\eqref{eizg} and $z=\theta+\psi$
with $\theta\in\mathcal{AP}^{n-2}$ and $\psi\in BC_0^{n-2}(\R,\C)$ (resp. $\psi\in
C_{00}^{n-2}(\R,\C))$), where $\theta$ and $\psi$ satisfy 
\begin{equation}\label{tpf}
\theta=G[\phi
+\varphi(\cdot,\Theta)]\qquad\text{and}\qquad\psi=G[g+h(\cdot,\theta)+R(\cdot,\Theta+\Psi)-R(\cdot,\Theta)]
\end{equation}
with $\Theta=(\theta,\theta',\dots,\theta^{(n-2)})$and $\Psi=(\psi,\psi',\dots,\psi^{(n-2)})$. Furthermore, if it holds \eqref{intgcond} for some $0<\beta<\min\{|\alpha_j|\ \mid 1\le j\le n-1\}$, then $z^{(i)}=O\big((I_\beta+I_{-\beta})[b]\big)$ for all $i=0,\dots,n-2$.
\end{Corollary}

\medskip
Now, we present a result which will be useful to study the problem
\eqref{generaleq} when $b$ belongs to this new class of almost periodic
functions $\lpaprc$ (resp. $\lpaporc$) given in Definition 2 and
$R$ is a suitable function that allows us to find solutions to
\eqref{generaleq} in this class of functions. See \cite[Lemma 9]{FP4}, for more details.

\begin{Lemma}\label{fzp}
Assume that $R\in\ml{AP}(\R\times\C^{n-1},\C)$ and satisfies \eqref{lipchitz}. If
$z\in\lpaprc$ (resp. $z\in\lpaporc$) then
$R(\cdot,Z(\cdot))\in\lpaprc$ (resp. $R(\cdot,Z(\cdot))\in
\ml{AP}_0(\R,\C,p)$). Furthermore, if $R=\varphi+h\in
\ml{AP}(\R\times\C^{n-1},\C,p)$ (resp. $R\in \ml{AP}_0(\R\times\C^{n-1},\C,p)$) with
$\varphi\in\ml{AP}(\R\times\C^{n-1},\C)$ satisfying \eqref{lipchitz} as above and
$h(\cdot,x)\in L^p_0$ (resp. $L^p(\R)$) uniformly for $x$
in compact subsets of $\C^{n-1}$ and $z\in\lpaprc$ (resp. $\lpaporc$), then $R(\cdot,Z(\cdot))\in\lpaprc$
(resp. $\lpaporc$).
\end{Lemma}

\medskip
In order to state the next result, let us denote for a suitable
given function $\theta$,
\begin{equation}\label{ft}
\tilde R_\theta(t,W)=R(t,\Theta+W)-R(t,\Theta),
\end{equation}
where $W=(w,w',\dots, w^{(n-2)})$. Moreover, assume that $\tilde R_\theta$ satisfies \eqref{lipchitz} with
$\zeta_M=\zeta_M(t)$ instead of $\xi_M$ for some $M>0$. Applying Proposition \ref{fixord} to equation for $\psi$ in \eqref{tpf}, namely, $T$ has a fixed point in $B=\{g\in BC^{n-2}\mid \|g\|_{BC^{n-2}}\le M\}$; we deduce the following fact. We omit its proof. 

\begin{Lemma}\label{spap}
Suppose that $b=\phi+g\in \ml{AP}(\R,\C,p)$ (resp. $\ml{AP}_0(\R,\C,p)$)
with $\phi\in\aprc$ and $g\in L^{p}_0$
(resp. $L^p(\R)$), $R(\cdot,0)=0$ and $R=\varphi +h\in
\ml{AP}(\R\times\C^{n-1},\C,p)$ (resp. $\ml{AP}_0(\R\times\C^{n-1},\C,p)$) with
$\varphi\in\ml{AP}(\R\times\C^{n-1},\C)$ and $h(\cdot,x)\in L^p_0$
(resp. $L^p(\R)$) uniformly for $x$ in compact subsets of $\C^{n-1}$, so
that, there is a solution $\theta\in\mathcal{AP}^{n-2}$ of the
integral equation
\begin{equation}\label{eit}
\theta=G [\phi+\varphi(\cdot,\Theta)].
\end{equation}
If 
\begin{equation*}
	\Big\| \sum_{j=1}^{n-1}\frac{\tilde\alpha_j}{\left|\Gamma_j\right|} I_{\alpha_j}\big[\zeta_M\big] \Big\|_\infty \leq \epsilon_0 \,, \quad\ \, 
	\bigg\| \sum_{i=0}^{n-2}\left|G\big[g+h(\cdot,\Theta)\big]^{(i)}\right|\bigg\|_\infty \leq(1-\epsilon_0)M\,,
\end{equation*}
then there is a solution $w$ of the integral equation
\begin{equation}\label{eiw}
w=G[g+h(\cdot,\Theta)+\tilde R_\theta(\cdot,W)]
\end{equation}
where $\tilde R_\theta$ is given by \eqref{ft} and $\theta$ satisfies \eqref{eit}. If, in addition, there exists $0<\beta<\min\{|\alpha_j| \mid 1\leq j\leq n-1  \}$ such that 
\begin{equation}\label{intgcond1}
	\Big\| \sum_{j=1}^{n-1}\frac{\tilde\alpha_j}{\left|\Gamma_j\right|} I_{\alpha_j-{\rm sgn}(\alpha_j)\beta}\big[\zeta_M\big] \Big\|_\infty < \frac{1}{2},
\end{equation}
then for all $i=0,1,\dots,n-2$
\begin{equation}\label{ew}
w^{(i)} =O\left((I_{\beta}+I_{-\beta})[g+h(\cdot,\theta)]\right)\qquad\text{and}\qquad w^{(i)} \in L^p_0\quad\text{(resp. $L^p(\R)$)}.
\end{equation}
Conversely, if $\theta$ and $w$ are such that $\theta^{(i)}\in\aprc$ and $w^{(i)}\in L^p_0$ (resp. $L^p(\R)$)
satisfy \eqref{eit} and \eqref{eiw} respectively, then
$z=\theta+w\in \ml{AP}^{n-2}(\R,\C,p)$ (resp.
$\ml{AP}_0^{n-2}(\R,\C,p)$) is a solution to \eqref{generaleq} and \eqref{eiz}
with $b=\phi+g$ and $R=\varphi+h$.
\end{Lemma}

\begin{Remark}
    {\rm Let us stress that the direct sum of $b$ and $R$ and the existence
of the ``almost periodic part'' $\theta$ induce the direct sum
of the solution $z$ and the equation \eqref{eizg} in the direct sum
\eqref{eit} with \eqref{eiw}, in these new classes of functions
$\lpaprc$ or $\lpaporc$. Furthermore, notice that we are not using
that $g$ and $h(\cdot,\theta)$ are bounded continuous functions,
but we need $G [g+h(\cdot,\theta)]\in BC(\R,\C)$. Notice
that an analogous version of Lemma \ref{spap} could be also
obtained for $b\in\aaprc$ or $\aaporc$ and
$R\in\ml{AAP}(\R\times\C^{n-1},\C)$ or $\ml{AAP}_0(\R\times\C^{n-1},\C)$. In
other words, we can study equation \eqref{eit} in $\aprc$ and then
equation \eqref{eiw} in either $BC_0(\R,\C)$ or $C_{00}(\R)$.}
\end{Remark}


\section{Existence of solutions for the Poincar\'e-Perron Problem}\label{espp}
We return to study Eq. \eqref{dif}, where $\mathcal{ D}$ is defined \eqref{linearz}, $P(r;\lambda)$ and $\mathcal{ L}$ are defined by \eqref{prl} and \eqref{onedee}, respectively, and $\mathcal{ F}$ is the non linear part of the equation given by \eqref{formf}.

Note that given $\omega\in\mathbb{C}$, $\omega\ne 0$ it holds
$$\int_{0}^tG_\omega[f](s)\,ds=-{1\over \omega}\int_{0}^t f(s)\,ds + {1\over \omega}\left\{G_\omega[f](t)-G_{\omega}[f](0)\right\}.$$
Hence, it follows that
  \begin{equation*}
      \begin{split}
          \int_{0}^tG[f](s)\,ds 
          =&\,(-1)^{n+1}\prod_{j=1}^{n-1}\gamma^{-1}_j \int_{0}^tf(s)\,ds  + \sum_{j=1}^{n-1}{1\over \Gamma_j\gamma_j}\left[G_{\gamma_j}[f](t)-G_{\gamma_j}[f](0)\right],
      \end{split}
  \end{equation*}
  in view of $\displaystyle\sum_{j=1}^{n-1}{1\over\Gamma_j\gamma_j}=(-1)^n \prod_{j=1}^{n-1}\gamma^{-1}_j$.
  
Taking into account Lemma \ref{dpol}, let $\gamma_j=\lambda_j-\lambda$, $j=1,\dots,n-1$ be the different roots of the characteristic polynomial $P_{\mathcal{D}}$ associated to the differential operator $\mathcal{ D}$ given by \eqref{linearz} and denote by $\alpha_j=\Re\gamma_j\not=0$\, and $\tilde{\alpha}_j=\sum_{i=0}^{n-2}|\gamma_j|^{i} $\,, for $j=1,\dots,n-1$. Note that in this situation $\Gamma_i=\prod_{k=1,k\ne i}^{n-1} (\lambda_i-\lambda_k)$ in view of \eqref{greenfor}. Recall that $\lambda$ is a fixed root of $P(a;x)$ and  $\tilde{\gamma}=\sum_{k=1}^{n-1}\frac{\tilde{\alpha}_k}{\left|\Gamma_k\right|}$. 

\medskip
\subsection{General results.} For the next result we shall estimate $\mathcal{ F}$ and $f_i$'s given by \eqref{formf} and \eqref{bimzi}, respectively. Denoting $X=(x_1,\dots,x_i)\in\R^i$ and $Y=(y_1,\dots,y_i)\in\R^i$\,, one has that
$$
\dfrac{\big|f_i(X)-f_i(Y)\big|}{\|X-Y\|} <\infty\,,
$$
for $X\not=Y$ in $\overline{B(0;\delta)}$\,, the closed ball of radius $\delta>0$ with centre in $0$\, in $\R^i$. It follows that the map
\begin{equation}\label{midef}
m_i:\delta \mapsto \sup\left\{\dfrac{\big|f_i(X)-f_i(Y)\big|}{\|X-Y\|}\mid X\not=Y\,;\,X,Y\in B(0;\delta)\right\}    
\end{equation}
is also continuous and non negative. Indeed, one has that there is a constant $C>0$ such that $m_i(\delta)\leq C\sup\{\|\nabla f_i(X)\| \mid X\in B(0;\delta) \} $. Moreover, since $f_i$ is a polynomial with terms of degree greater than 1, one obtains that $m_i(\delta)\to 0$ for $\delta\to 0$ and $m_i(\delta)\to \infty$ for $\delta\to \infty$. Thus, we have the following result.

\begin{Lemma}
Define $m(\delta)=\max_{i=2,\dots,n}\{m_i(\delta)\}\,$, where $m_i$ are defined by equation \eqref{midef}. Then, $m$ extends to a continuous map defined on $[0,\infty)$, such that $m(0)=0$, which satisfies $m(\delta)>0$ for all $\delta>0$\,, $m(\delta)\to +\infty$ as $\delta\to+\infty$  and for all $i=1,\dots,n-1$
\begin{equation}\label{mofm}
\big|f_i(X)-f_i(Y) \big|\leq m(\delta) \| X - Y \| \,,\quad \text{ for all } \, X,\,Y\in \overline{B(0;\delta)} \,.
\end{equation}       
\end{Lemma}

Then, for $\delta>0$ and $z_1,\,z_2\in \mathcal{AP}^{n-2}(\R)$ such that $\|z_i\|_{\mathcal{AP}^{n-2}}<\delta$\,, for $i=1,2$\,, we have that 
\begin{equation}\label{flipzt}
\begin{split}
\big|\mathcal{ F}\big(t,Z_1(t)\big) -\mathcal{ F}\big(t,Z_2(t)\big) \big|  &\leq m(\delta) \|Z_1(t)-Z_2(t)\|  \sum_{i=2}^{n}\sum_{j=0}^{n-i}{i+j \choose j} 
\big(|a_{i+j}|+|r_{i+j}(t)|\big)|\lambda|^j  
\end{split}
\end{equation}

Now, denote for $0\le \beta< \min\{|\alpha_j| \mid 0\leq j\leq n-1\}$
\begin{equation}\label{l}
\begin{split}
	L_\beta &:=\sum_{j=1}^{n-1}\sum_{k=1}^{n-1} \left|\frac{\tilde{\alpha}_j}{\Gamma_j k! }\right| \left\|  I_{\alpha_j-\sgn(\alpha_j)\beta}\Big[\frac{\partial^k}{\partial x^k}P(r(\cdot);\lambda)\Big]\! \right\|_\infty \,, 
\end{split}
	\end{equation} 
and
 \begin{equation}\label{q}
\begin{split}
Q_\beta &:= \sum_{j=1}^{n-1} \sum_{i=2}^{n} \sum_{k=0}^{i-2}\frac{\tilde{\alpha}_j}{\left|\Gamma_j\right|} {i\choose k}|\lambda|^{k} \Big( {|a_i|\over |\alpha_j-{\rm sgn}(\alpha_j)\beta|}+ \left\| I_{\alpha_j-\sgn(\alpha_j)\beta}\big[|r_i|\big] \right\|_\infty\Big) \,.
\end{split}
	\end{equation}
Hence, assuming $L_0<1$, let $\delta_\lambda>0$ be such that 
\begin{equation}\label{delta0}
    0<m(\delta)<\dfrac{1-L_0}{Q_0}\quad\text{ for all $\delta\in(0,\delta_\lambda)$ and } m(\delta_\lambda)=\dfrac{1-L_0}{Q_0}. 
\end{equation}

Our main result reads as follows.

\medskip

\begin{Theorem}\label{orederzun}
	Consider the equation \eqref{dif} with $r_i\in\ap(\R)$\,, for $i=1,\dots, n-1$. Let $\gamma_j$ be the different roots of the differential operator $\mathcal{ D}$ given by \eqref{linearz} and denote by $\alpha_j=\Re\gamma_j\not=0$\,, for $j=1,\dots,n-1$. Let us also denote by $\tilde{\alpha}_j=\sum_{i=0}^{n-2}|\gamma_j|^{i} $\,. Assume that 
	\begin{equation}\label{condeq}
		L_0<1\quad\text{ and }\quad H \le \max\{[1-L_0-m(\delta)Q_0]\delta \ \mid \ \delta\in(0,\delta_\lambda)\}\,,
	\end{equation}
where $L_0$ and $Q_0$ are given by \eqref{l}  and \eqref{q}  with $\beta=0$, respectively, and
\begin{equation}
   H:= \left\| \sum_{i=0}^{n-2}\left|G[P(r;\lambda)]^{(i)}\right|\right\|_\infty \,.
\end{equation}
Then, there exists a solution $y$ for the equation \eqref{poinca1} such that
\begin{equation}\label{fath1}
\begin{split}
    y_\lambda(t)=&\, e^{\lambda t} \exp\bigg((-1)^n\prod_{j=1}^{n-1}\gamma^{-1}_j \int_{0}^t\big[P(r;\lambda)(s)+\mathcal{L}(s,z(s))+\mathcal{F}(s,Z(s))\big]\,ds\bigg)\\
    & \times\exp\bigg(- \sum_{j=1}^{n-1}{1\over \Gamma_j\gamma_j} G_{\gamma_j}\big[P(r;\lambda)+\mathcal{L}(\cdot,z)+\mathcal{F}(\cdot,Z)\big](t)\bigg),
\end{split}
\end{equation}
where $z\in\mathcal{AP}^{n-2}(\mathbb{R})$ satisfy the differential equation \eqref{dif} and integral equation
\begin{equation}\label{eiz}
    z= - G[P(r;\lambda) + \mathcal{L}(\cdot, z) + \mathcal{F}(\cdot,Z)].
\end{equation}
In addition, if there exists a constant $0<\beta< \min\{|\alpha_j| \mid 1\leq j\leq n-1\}$ such that 
\begin{equation}\label{cl}
L_\beta <\frac{1}{2},
\end{equation}
then $z$ satisfies $z^{(i)}=O\big((I_{\beta}+I_{-\beta})[P(r;\lambda)]\big)$ for $i=0,1,\dots,n-2$\,. Moreover \begin{equation}\label{yjioyi}
    \frac{y_\lambda^{(i)}(t)}{y_\lambda(t)}=\lambda^{i} + O\big((I_{\beta}+I_{-\beta})[P(r;\lambda)]\big)
\end{equation}
\end{Theorem}

\begin{proof} We shall apply Proposition \ref{fixord} to equation \eqref{dif} by taking $b(t)=-P(r(t);\lambda)$ and \linebreak$R(t,z,z',\dots,z^{(n-2)})=-\mathcal{L}(t,z)-\mathcal{F}(t,z,z',\dots,z^{(n-2)})$, so that, $b\in\ap$ and $R(\cdot, Z)\in \mathcal{AP}(\R\times\C^{n-1})$, see \eqref{onedee} and \eqref{formf}. 

Arguing as in the proof of \cite[Theorem 3.1]{BFP}, it follows that $R=-\mathcal{L}-\mathcal{F} $ satisfy the Lipschitz condition  \eqref{lipchitz} with the choice
\begin{equation}\label{xisubm}
\begin{split}
\xi_{\delta}(t)&= \sum_{k=1}^{n-1} \frac{1}{k!}\bigg|\frac{\partial^k}{\partial x^k} P(r(\cdot);\lambda)\bigg|+  m(\delta)\sum_{i=2}^{n}\sum_{j=0}^{i-2}{i\choose j}  \big(|a_{i}|+|r_{i}(t)|\big)|\lambda|^j
\end{split}
\end{equation}
and
$$\sum_{j=0}^{n-1}\left|\frac{\tilde{\alpha}_j}{\Gamma_j}\right| I_{\alpha_j-{\rm sgn}(\alpha_j)\beta}\big[\xi_{\delta}\big](t)  \leq  m(\delta)Q_\beta +L_\beta\,.$$
Note that for any function $f$ it holds $I_{\alpha_j}[f](t)\le I_{\alpha_j-\sgn(\alpha_j)\beta}[f](t),$ so that, $L_0\le L_\beta$ and $Q_0\le Q_\beta$. Hence, by definition of $\delta_\lambda>0$ and \eqref{delta0} it follows that function $g_\lambda(\delta)=[1-L_0-m_\lambda(\delta)Q_0]\delta$ satisfies $g_\lambda(0)=0=g_\lambda(\delta_\lambda)$ and $g_\lambda(\delta)>0$ for all $\delta\in(0,\delta_\lambda)$. Then, we choose $M\in(0,\delta_\lambda)$ satisfying
\begin{equation}\label{M}
g_\lambda(M)=[1-L_0-m(M)Q_0]M=\max\{[1-L_0-m(\delta)Q_0]\delta \ \mid \ \delta\in(0,\delta_\lambda)\}.
\end{equation}
Therefore, we have that $m(M)<\frac{1}{Q_0}(1-L_0)$\, and we choose $\epsilon_0=m(M)Q_0+L_0<1$. Now, by the choice of $M$ in \eqref{M}, it follows that $H$ satisfies 
\begin{equation*}
	 H  \leq [1-L_0-m_\lambda(M)Q_0]M=[1-(m_\lambda(M)Q_0+L_0)]M = (1-\epsilon_0)M 
\end{equation*}
Hence, the conditions of Proposition \ref{fixord} are satisfied, thus there exists a solution $z\in \mathcal{AP}^{n-2}$ of the integral equation \eqref{eiz} which implies that $z$ satisfies Eq. \eqref{generaleq} that corresponds in this case to \eqref{dif}. Hence, $y_\lambda$ given by  \eqref{yeilz} is a solution to equation \eqref{poinca1}. Moreover, integrating \eqref{eiz} we find that
{\small  \begin{equation*}
      \begin{split}
          \int_{0}^t z(s)\,ds =&\,(-1)^n\prod_{j=1}^{n-1}\gamma^{-1}_j \int_{0}^t\big[P(r;\lambda)(s)+\mathcal{L}(s,z(s))+\mathcal{F}(s,Z(s))\big]\,ds \\
          &\,- \sum_{j=1}^{n-1}{1\over \Gamma_j\gamma_j}\left\{G_{\gamma_j}\big[P(r;\lambda)+\mathcal{L}(\cdot,z)+\mathcal{F}(\cdot,Z)\big](t) + \text{const.}\right\} 
      \end{split}
  \end{equation*}}
Thus, we conclude \eqref{fath1}. Finally, by using \eqref{cl} and again arguing as in the proof of \cite[Theorem 3.1]{BFP} there exists $M$ small enough such that  $m(M)<\frac{1}{Q_\beta}({1\over 2}-L_\beta)$ and condition \eqref{intgcond} in Proposition \ref{fixord} is satisfied. Therefore, the solution $z$ of equation  \eqref{dif} and also satisfies
$z^{(i)}=O\big( (I_{\beta}+I_{-\beta})\big[P(r;\lambda)] \big)$\, for $i=0,\dots,n-2$\,. 
Moreover, from \eqref{ynbell}  we get immediately \eqref{yjioyi}. This finishes the proof.
\end{proof}

\begin{Remark}\label{mdelta}
    {\rm Let us stress that inequality \eqref{condeq} is an explicit sufficient condition in terms of $H$, $L_0$ and $Q_0$. The choice of $m(\delta)$ lead us to obtain $M=M(\delta_\lambda)$ satisfying $g_\lambda(M)=\max\{g_\lambda(\delta)\, :\, \delta\in(0,\delta_\lambda)\}$.	Notice that we can replace $m(\delta)$ in \eqref{mofm}, \eqref{flipzt} and \eqref{condeq} for any increasing function, say $\tilde m(\delta)$, satisfying $\tilde m(0)=0$ and $m(\delta)\le \tilde m(\delta)$.    However, assumption on $H$ becomes stronger since
 $$\max\{[1-L_0-\tilde{m}(\delta)Q_0]\delta\mid \delta\in (0,\tilde \delta_\lambda)\} \le \max\{[1-L_0- m(\delta)Q_0]\delta\mid \delta\in (0, \delta_\lambda)\}$$
 with $\tilde{\delta}_\lambda\le \delta_\lambda$ and $\tilde m(\tilde\delta_\lambda)=\dfrac{1-L_0}{Q_0}$.
    }
\end{Remark}

\medskip
Note that $L_0$, $Q_0$ and $H$ in the hypothesis of Theorem \ref{orederzun} depend of the election of $\lambda$. If we want to put in evidence this dependency we rewrite 
 $L_0^\lambda$, $Q_0^\lambda$ and $H_{\lambda}$. 

Now, we shall prove the existence of fundamental system of solutions $\{y_i\}_i$, $i=1,\dots, n$, assuming conditions of Theorem \ref{orederzun} for each $\lambda_i$, $i=1,\dots, n$, the roots of $P(a;x)$. For simplicity, we will only consider assumption \eqref{cl}. In this way, we recover Perron's result \cite{Perr}. To this purpose, we shall denote $P(r;\lambda_i)$, $\mathcal{D}_i$, $P_{\mathcal{D}_i}$ and $\mathcal{F}_i$ by \eqref{prl}, \eqref{linearz}, \eqref{defdpol} and \eqref{formf} with $\lambda=\lambda_i$ respectively.  We shall assume that $\Re(\lambda_i)\ne \Re(\lambda_j)$ for all $i\ne j$. Also, consider $\mathcal{N}(i):=\{1,\dots,n\}\setminus \{i\}$ and the bijection $\sigma_i:\mathcal{N}(i)\to \{1,\dots,n-1\}$ defined by
$$\sigma_i(j)=\begin{cases}
j&\text{if }j<i\\
j-1&\text{if }i<j
\end{cases}\qquad\text{and}\qquad \sigma_i^{-1}(j)=\begin{cases}
j&\text{if }j<i\\
j+1&\text{if }i<j
\end{cases}.$$
Thus, fixing $i\in\{1,\dots,n\}$ we have that $\gamma_{ij}=\lambda_{\sigma_i^{-1}(j)}-\lambda_i$, $j=1,\dots,n-1$ are the roots of polynomial $P_{\mathcal{D}_i}$ as shown in Lemma \ref{dpol}. Furthermore, Green's operator for equation $\mathcal{D}_iz=f$ is given by
$$
G^{\lambda_i}\big[f\big](t):=\int_{-\infty}^{\infty} \mathcal{G}_i(t,s)f(s)\,ds = \sum_{j=1}^{n-1}\frac{1}{\Gamma_{ij}}G_{\gamma_{ij}}\big[f\big](t)\,,
$$
where

\begin{equation*}
\Gamma_{ij}:= \prod_{k= 1,\,k\not=j}^{n-1} (\gamma_{ij}-\gamma_{ik})\quad\text{and}\quad \mathcal{G}_i(t,s)= \sum_{j=1}^{n-1}\frac{1}{\Gamma_{ij} }g_{\gamma_{ij}}(t,s),
\end{equation*}
with $g_\gamma$ is given by \eqref{defg}. Denote $\alpha_{ij}=\Re(\gamma_{ij})$ and  $\displaystyle\tilde\alpha_{ij}=\sum_{k=1}^{n-2}|\gamma_{ij}^k|$. For $k\in\{1,\dots,n\}$, let $\delta_{\lambda_k}>0$ be such that 
\begin{equation}\label{deltalk}
0< m(\delta)<\dfrac{1-L_0^{\lambda_k}}{Q_0^{\lambda_k}}\quad\text{for all $\delta\in(0,\delta_{\lambda_k})$ and } m(\delta_{\lambda_k})=\dfrac{1-L_0^{\lambda_k}}{Q_0^{\lambda_k}}.    
\end{equation}
For the next result and some consequences, we consider the submultiplicative norms for a matrix $B=(b_{ik})$ defined by 
$$
\|B\|_1=\max_{k}\sum_i^n|b_{ik}|
\qquad\text{and}\qquad \|B\|_\infty=\max_{i}\sum_k^n|b_{ik}|.$$ 

\begin{Theorem}\label{base}
	Consider the equation \eqref{dif} with $r_i\in\ap(\R)$\,, for $i=1,\dots, n-1$. Let $\lambda_1,\dots,\lambda_n$ be the $n$ roots of the polynomial $P(a,x)=0$ and $V$ the Vandermonde matrix of $\lambda_1,\dots,\lambda_n$. Assume that for every $k=1,\dots, n$ there is a constant $M_k\in (0,\min\limits_k \delta_{\lambda_k})$ such that 
$$
 L_0^{\lambda_k}<1\,, \quad \text{ and }\quad H_{\lambda_k}\leq \big[1-L_0^{\lambda_k}-m(M_k)Q^{\lambda_k}_0\big]M_k
 $$
and
\begin{equation}\label{conMbase}
\max_k \left(M_k\big(m(M_k)+1\big)\sum_{i=2}^{n}\big[(|\lambda_k|+1)^{i-1}-|\lambda_k|^{i-1}\big]\right)\! \|V^{-1}\|_1 <1. 
\end{equation}
Then there exists a fundamental system of solutions $y_k$, $k=1,\dots,n$ to equation \eqref{poinca1} such that
\begin{equation*}
\begin{split}
    y_{k}(t)=&\, e^{\lambda_k t}\exp\bigg(- \sum_{j=1}^{n-1}{1\over \Gamma_{kj}\gamma_{kj}} G_{\gamma_{kj}}\big[P(r;\lambda_k)+\mathcal{L}_k(\cdot,z_k)+\mathcal{F}_k(\cdot,Z_k)\big](t)\bigg)\\
    &\times \exp\bigg((-1)^n\prod_{j\in \mathcal{N}(k)}(\lambda_j-\lambda_k)^{-1} \int_{0}^t\big[P(r;\lambda_k)(s)+\mathcal{L}_k(s,z_k(s))+\mathcal{F}_k(s,Z_k(s))\big]\,ds\bigg), \end{split}
\end{equation*}
where $\mathcal{L}_i(\cdot,z_i)$ and $\mathcal{F}_i(\cdot,Z_i)$ are given by \eqref{onedee} and \eqref{formf} with $\lambda=\lambda_i$, respectively, $z_i$ satisfies differential equation
$$\mathcal{D}_iz+P(r;\lambda_i)+\mathcal{L}_i(t,z)+\mathcal{F}_i\big(t,z,\dots,z^{(n-2)}\big)=0,$$
the integral equation
$$z=- G^{\lambda_i}[P(r;\lambda_i) + \mathcal{L}_i(\cdot,z)+\mathcal{F}_i(\cdot,Z)]$$
and $z_i\in \mathcal{AP}^{n-2}(\mathbb{R})$ for $i=1,\dots,n$.
	\end{Theorem}
\begin{proof}
First, by using the choice of $M_k$ and the definition of $\delta_{\lambda_k}$, one has that   
$$
0<\big[1-L_0^{\lambda_k}-m(M_k)Q_0^{\lambda_k}\big]M_k\leq \max\{[1-L_0^{\lambda_k}-m(\delta)Q_0^{\lambda_k}]\delta \ \mid \ \delta\in(0,\delta_{\lambda_k})\}\,,
$$
for each $k$, where $\delta_{\lambda_k}$ satisfies \eqref{deltalk}. It follows that the hypothesis of Theorem \ref{orederzun} holds, then there exists a solution $y_k$ of equation \eqref{poinca1} for $\lambda=\lambda_k$, $k=1,\dots,n$.

It remains to prove that, $\{y_{1},\dots ,y_{n}\}$ are linearly independent.	We shall show that the Wronskian of those functions, denoted by $W$, does not vanish. By equation \eqref{ynbell}, one obtains for $i\geq 1$ that
$$
 \begin{aligned}
     \frac{y_k^{(i)}(t)}{y_k(t)} 
     &= \lambda_k^{i}+\sum_{j=1}^{i}{i \choose j} \lambda_k^{i-j} \Big(f_{j-1}\big(z_k(t),z_k'(t),\dots,z_k^{(j-2)}(t)\big)+z_k^{(j-1)}\Big)\,.
 \end{aligned}
$$
where the polynomials $f_j$ are defined in \eqref{bimzi}. Then, denoting by $A=(a_{ik})$ the matrix with entries 
$$
a_{ik}=\begin{cases}
    \sum\limits_{j=1}^{i-1}{i-1 \choose j} \lambda_k^{i-1-j} \Big(f_{j-1}\big(z_k(t),z'_k(t),\dots,z_k^{(j-2)}(t)\big)+z_k^{(j-1)}\Big) & i\geq 2 \\
    0 & i=1
\end{cases}\,,
$$
$W$ can be rewritten by 
$$
W = (y_1y_2\cdots y_n )  \det \left[V(I+A)\right]=(y_1y_2\cdots y_n )  \det[V]\det[I+V^{-1}A]\,, 
$$
where $V$ is the Vandermonde matrix of $\lambda_1,\dots,\lambda_k$. Thus, it is enough to show that $I+V^{-1}A$ is an invertible matrix. It is well known that if  $\|A\|_1\|V^{-1}\|_1<1$ then the matrix $I+V^{-1}A$ is invertible.

We now estimate the $\|A\|_1$. Since $f_j(0,\dots,0)=0$ for all $j$, and using inequality \eqref{mofm}, we obtain that  
$$
\begin{aligned}
|a_{ik}| &\leq \sum_{j=1}^{i-1}{i-1 \choose j} |\lambda_k|^{i-1-j} \big(m(M_k)\|Z_k\|+\|Z_k\|\big) \\
&\leq M_k(m(M_k)+1)\sum_{j=1}^{i-1}{i-1 \choose j} |\lambda_k|^{i-1-j} \\
&\leq M_k(m(M_k)+1)\big[(|\lambda_k|+1)^{i-1}-|\lambda_k|^{i-1} \big] 
\end{aligned}
$$
for $\|Z_k\|\leq M_k$ and $i>1$ with $Z_k=(z_k,z_k'\dots,z_k^{(n-2)})$.
Notice that, if $\lambda_k=0$ then $|a_{ik}|\leq  M_k(m(M_k)+1)$ for all $i$, so in general one has 
$$
\sum\limits_{i=1}^n|a_{ik}|=\sum\limits_{i=2}^n|a_{ik}| 
\leq M_k(m(M_k)+1) \sum_{i=2}^{n} \big[(|\lambda_k|+1)^{i-1}-|\lambda_k|^{i-1} \big]
$$
Thus, we obtain that
$$
\begin{aligned}
\|A\|_1  &\leq \,\max_k \left( M_k(m(M_k)+1) \sum_{i=2}^{n} \big[(|\lambda_k|+1)^{i-1}-|\lambda_k|^{i-1} \big]\right).
\end{aligned}
$$
Therefore, by hypothesis, one obtains that $\|A\|_1\|V^{-1}\|_1<1$\,, finishing the proof.
\end{proof}

In \cite{GW} it was shown that
$$
\|V^{-1}\|_\infty\leq \max_{k}\prod_{l=1,\,l\not=k}^n\frac{1+|\lambda_l|}{|\lambda_k-\lambda_l|} 
$$
where equality holds if (but not only if) all complex numbers $\lambda_k$, are on the same ray through the origin. 
It is well known that $\|B\|_1\leq n\|B\|_\infty$. Hence, a different sufficient condition to obtain also the previous result, but avoiding the computation of $V^{-1}$, is
 $$
n\, \max_k \left( M_k\big(m(M_k)+1\big)\sum_{i=2}^{n}\big[(|\lambda_k|+1)^{i-1}-|\lambda_k|^{i-1}\big]\right)\! \max \left(\prod_{l=1,l\not=k}^n\frac{1+|\lambda_l|}{|\lambda_k-\lambda_l|} \right) <1\,.
 $$


Another kind of perturbations are $r_i\in \aapr$ for $i = 0,\dots,n-1$. Let us recall that $\ml{AAP}^{n-2}(\R, \C)$ is a closed subspace of $BC^{n-2}(\R, \C)$. From the ideas used in  Theorem \ref{orederzun} and Corollary \ref{o2} it readily follows the next result. In particular, the decomposition of the coefficients $r_i$ , $i = 0,\dots,n-1$ induces the direct sum of the solution $z$. Direct computations from the definition of $P(r;\lambda)$, $\ml L(t,z)$ and $\ml F(t,Z)$ in \eqref{prl}, \eqref{onedee} and \eqref{formf}, respectively, binomial property of Bell's polynomials \eqref{reducedz} and assuming that $r_i =\mu_i+\nu_i$\,, $i=0,\dots,n-1$ where $\mu_i\in\ap(\R)$ and $\nu_i\in BC_0(\R)$\,(resp. $C_{00}(\R)$) lead us to obtain that
$$P(r;\lambda)=P(\mu;\lambda)+P(\nu;\lambda),$$
with $\mu=(\mu_0,\dots,\mu_n)$, $\mu_n\equiv 0$, $\nu=(\nu_0,\dots,\nu_n),\quad \nu_n\equiv 0$,
$$\ml L(\cdot,\theta+\psi)=\ml L_\mu(\cdot,\theta) + \ml L_\nu(\cdot,\theta) + \ml L(\cdot,\psi)$$
and
\begin{equation}\label{ftmphi}
\ml F(\cdot,\Theta+\Psi)= \ml F_\mu(\cdot,\Theta) + \tilde{\ml F}_\nu(\cdot,\Theta) + \tilde{\ml L}_\theta(\cdot,\psi) + \tilde{\ml F}_\theta(\cdot,\Psi),    
\end{equation}
where
\begin{equation}
    \ml L_\mu(t,\theta)=\sum_{k=1}^{n-1}\frac{1}{k!} \frac{\partial^k}{\partial x^k}[P(\mu;\lambda)]\theta^{(k-1)}(t),\quad
    \ml L_\nu(t,\theta)=\sum_{k=1}^{n-1}\frac{1}{k!} \frac{\partial^k}{\partial x^k}[P(\nu;\lambda)]\theta^{(k-1)}(t),
\end{equation}
\begin{equation}
\mathcal{F}_\mu(t,\Theta)=\sum_{i=2}^{n}\sum_{j=0}^{n-i}{i+j \choose j}(a_{i+j}+\mu_{i+j}(t))\lambda^j\big(B_i(\theta,\theta',\dots,\theta^{(i-1)})-\theta^{(i-1)}\big)\,,
\end{equation}
\begin{equation}
\tilde{\mathcal{F}}_\nu(t,\Theta)=\sum_{i=2}^{n}\sum_{j=0}^{n-i}{i+j \choose j}\nu_{i+j}(t)\lambda^j\big[B_i(\theta,\theta',\dots,\theta^{(i-1)})-\theta^{(i-1)}\big]\,,
\end{equation}
\begin{equation}
\tilde{\mathcal{L}}_\theta(t,\psi)=\sum_{i=2}^{n}\sum_{j=0}^{n-i}{i+j \choose j}(a_{i+j}+ r_{i+j}(t))\lambda^j\sum_{k=1}^{i-1} {i\choose k} B_{i-k}(\theta,\theta',\dots,\theta^{(i-k-1)})\psi^{(k-1)}\,
\end{equation}
and
\begin{equation}
\begin{split}
\tilde{\mathcal{F}}_\theta(t,\psi)
=&\,\sum_{i=2}^{n}\sum_{j=0}^{n-i}{i+j \choose j}(a_{i+j}+r_{i+j}(t))\lambda^j \\
&\times \sum_{k=2}^{i} {i\choose k} B_{i-k}(\theta,\dots,\theta^{(i-k-1)})\big[B_k(\psi,\dots,\psi^{k-1}) -\psi^{(k-1)}\big]\,.    
\end{split}
\end{equation}
Notice that $P(\mu;\lambda),\frac{\partial^k}{\partial x^k}[P(\mu;\lambda)]\in\ap$ for all $k=1,\dots, n-1$ and $\ml L_\mu(\cdot,\theta),\ml F_\mu(\cdot,\Theta)\in\ap$ if $\theta\in\ap^{n-2}$, in view of $B_i(\theta,\theta',\dots,\theta^{(i-1)}) \in \ap$ for all $i=1,\dots,n$. Furthermore, for any $z\in BC^{n-2}$, $\ml L(\cdot,z)=\ml L_\mu(\cdot,z) + \ml L_\nu(\cdot,z)$ and $\ml F(\cdot,z)=\ml F_\mu(\cdot,z) + \tilde{\ml F}_\nu(\cdot,z)$.

 \begin{Theorem}\label{teoczero}
 	Consider the equation \eqref{poinca1} with $r_i\in\aap(\R)$\,(resp. $\aapor$), $i=0,\dots,n-1$. Assume that \eqref{condeq} holds. Then, there exists a solution $y$ to equation \eqref{poinca1} satisfying \eqref{fath1}, where $z$ is a solution to differential equation \eqref{dif} and $z^{(i)}\in\aap$ (resp. $\aapor$) for $i=0,1,\dots,n-2$. Moreover, if $r_i =\mu_i+\nu_i$\,, $i=0,\dots,n-1$ where $\mu_i\in\ap(\R)$ and $\nu_i\in BC_0(\R)$\,(resp. $C_{00}(\R)$), then $z=\theta+\psi$, where $\theta\in\ap^{n-2}(\R)$ and $\psi\in BC_0^{n-2}(\R)$ (resp. $C_{00}^{n-2}(\R)$) satisfy
\begin{equation}\label{thetaap}
    \theta=-G[P(\mu;\lambda) + \ml L_\mu(\cdot,\theta) + \ml F_\mu(\cdot,\Theta)]
\end{equation}
and
\begin{equation}\label{eqpsi}
    \psi=-G[P(\nu;\lambda) +\ml L_\nu(\cdot,\theta)+\tilde{\ml F}_\nu(\cdot,\Theta) + \ml L(\cdot,\psi) +\tilde{\ml L}_\theta(\cdot,\Psi) + \tilde{\ml F}_\theta(\cdot,\Psi)]
\end{equation}
 	Conversely, if $\theta\in\ap(\R)$ and $\psi\in BC_0(\R)$ (resp. $C_{00}(\R)$) satisfy equations \eqref{thetaap} and \eqref{eqpsi}, respectively, then $z=\theta+\psi$ is solution of  equation \eqref{poinca1} in $\aapr$ (resp.$\aapor$)\,.
 \end{Theorem}
 
\begin{proof}
	It is enough to prove the decomposition of the equation \eqref{eiz} in these two equations \eqref{thetaap} and \eqref{eqpsi}. Note that, $P(r,\lambda)$ and ${\mathcal L}$ are linear in the variable $r$\,. 

 On the other hand, from property \eqref{reducedz} it follows that
 \begin{equation*}
   \begin{split}
       B_i(\theta+\psi,\dots,\theta^{(i-1)}+\psi^{(i-1)})=&\, B_i(\theta,\dots,\theta^{(i-1)}) \\
       & + \sum_{k=1}^{i}{i \choose k} B_{i-k}\big(\theta,\dots,\theta^{(i-1)}\big)B_k(\psi,\dots,\psi^{(k-1)})
   \end{split}  
 \end{equation*}
 so that
 \begin{equation}
     \begin{split}
         &B_i(\theta+\psi,\dots,\theta^{(i-1)}+\psi^{(i-1)})-[\theta^{(i-1)} +\psi^{(i-1)}]\\
         & =B_i(\theta,\dots,\theta^{(i-1)}) -\theta^{(i-1)} + \sum_{k=1}^{i-1}{i \choose k} B_{i-k}(\theta,\dots,\theta^{(i-1)})\psi^{(k-1)}\\
         &\quad +\sum_{k=1}^{i}{i \choose k} B_{i-k}(\theta,\dots,\theta^{(i-1)})\big[B_k(\psi,\dots,\psi^{(k-1)})-\psi^{(k-1)}\big].
     \end{split}
 \end{equation}
 Hence, we deduce that
 \begin{equation}
     \begin{split}
         &(a_{i+j} + \mu_{i+j} + \nu_{i+j}) \Big(B_i(\theta+\psi,\dots,\theta^{(i-1)}+\psi^{(i-1)})-[\theta^{(i-1)} +\psi^{(i-1)}]\Big)\\
         & = (a_{i+j} + \mu_{i+j}) \big[B_i(\theta,\dots,\theta^{(i-1)}) -\theta^{(i-1)}\big] +\nu_{i+j} \big[B_i(\theta,\dots,\theta^{(i-1)}) -\theta^{(i-1)}\big] \\
         & \quad +(a_{i+j} + \mu_{i+j} + \nu_{i+j}) \sum_{k=1}^{i-1}{i \choose k} B_{i-k}(\theta,\dots,\theta^{(i-1)})\psi^{(k-1)}\\
         &\quad +(a_{i+j} + \mu_{i+j} + \nu_{i+j})\sum_{k=1}^{i}{i \choose k} B_{i-k}(\theta,\dots,\theta^{(i-1)})\big[B_k(\psi,\dots,\psi^{(k-1)})-\psi^{(k-1)}\big].
     \end{split}
 \end{equation}
Therefore, we conclude \eqref{ftmphi}. This finishes the proof.
\end{proof}

\begin{Remark}
	{\rm
	In the same spirit of \cite[Theorems 3 and 4]{FP4}, we could perturb $r_i\in\aprc$, $i=0,\dots, n-1$ with $L^p$-functions in equation \eqref{poinca1}, see \cite{BFP,CHP2,FP1,FP3}. In other words, equation \eqref{poinca1} with $r_i\in\ap(\R,\C,p)$, $i=0,\dots, n-1$. Let us stress that $L_0^p\cap BC(\R,\C)$ and $L^p(\R)\cap BC(\R,\C)$ are not closed subspaces of $BC(\R,\C)$, so that, we cannot obtain readily a version of Theorem \ref{teoczero} in the subspaces $\ap(\R,\C,p)$ and $\ap_0(\R,\C,p)$. However, we can find such a solution $z$ in $\ap(\R,\C,p)$ or $\ap_0(\R,\C,p)$ by exploiting the decomposition property of coefficients $r_i=\mu_i+\nu_i$ ($\mu_i\in\ap(\R)$ and either $\nu_i\in BC_0(\R)$ or $C_{00}(\R)$ or $L_0^p$ or $L^p(\R)$), $i=0,\dots, n-1$, studying equations \eqref{thetaap} and \eqref{eqpsi}, and applying Proposition \ref{fixord} to \eqref{thetaap} and then Lemma \ref{spap}. For simplicity, we present a result only in case $n=3$ and next subsection. We omit its proof.}
\end{Remark}

\subsection{The case $n=3$}
In this subsection we study Theorem \ref{orederzun} and present some consequences of previous result (Theorem \ref{teoczero}) in case $n=3$, namely,
\begin{equation}\label{eqn3}
    y'''+(a_2+r_2(t))y''+(a_1+r_1(t))y'+(a_0+r_0(t))y=0.
\end{equation}
We assume that $\lambda$, $\lambda_1$ and $\lambda_2$ are the roots of $P(a;x)=x^3+a_2x^2+a_1x+a_0$, so that $\gamma_i=\lambda_i-\lambda$, $\alpha_i=\Re\gamma_i$, $i=1,2$, $\tilde\alpha_j=1+|\gamma_j|$ and $\Gamma_i=(-1)^{i}(\lambda_2-\lambda_1)$. Notice that $P(r;\lambda)$, $\mathcal{L}(\cdot,z)$ and $\mathcal{F}(\cdot,z,z')$ are given by \eqref{prn3}-\eqref{fzn3}, respectively. It is readily checked that $f_1(x_1)= x_1^2$ and $f_2(x_1,x_2) = x_1^3+3x_1x_2$.    Hence, it follows that for any $\delta>0$
    \begin{itemize}
        \item $|x|\le \delta$ and $|y|\le \delta$ implies $|f_1(x)-f_1(y)|\le 2\delta |x-y|$;
        \item $|x_1|+|x_2|\le \delta$ and $|y_1|+|y_2|\le \delta$ implies 
        $$|f_2(x_1,x_2)-f_2(y_1,y_2)|\le (3\delta^2+3\delta)\big[|x_1-y_1|+x_2-y_2|\big].$$
    \end{itemize}
Thus, according to Remark \ref{mdelta} we can choose $m(\delta)=\max\{2\delta,3\delta^2+3\delta\}=3\delta^2+3\delta$ to apply Theorem \ref{orederzun}. Furthermore, we obtain that for $0<\beta<\min\{|\alpha_1|,|\alpha_2|\}$
$$L_\beta=\sum_{i=1}^2\frac{1+|\gamma_i|}{|\lambda_1-\lambda_2|}\Big[ \big\|I_{\alpha_i-{\rm sgn}(\alpha_i)\beta}[2\lambda r_2+r_1]\big\|_\infty + \big\|I_{\alpha_i-{\rm sgn}(\alpha_i)\beta}[r_2]\big\|_\infty\Big]$$
and
$$Q_\beta=\sum_{i=1}^2\frac{1+|\gamma_i|}{|\lambda_1-\lambda_2|}\bigg[ \frac{1+3|\lambda|+|a_2|}{|\alpha_i-{\rm sgn}(\alpha_i)\beta|} + \big\|I_{\alpha_i-{\rm sgn}(\alpha_i)\beta}[r_2]\big\|_\infty\bigg].$$
Recall that $a_3=1$ and $r_3\equiv 0$. Now, assume that $L_0<1$ and consider the function $g(\delta)=[1-L_0-(3\delta^2+3\delta)Q_0]\delta$. It is readily checked that $\delta_\lambda=\frac{1}{2}\Big(\sqrt{1+\frac{4(1-L_0)}{3Q_0}} - 1\Big)>0$ satisfies
$3Q_0\delta^2 + 3Q_0 \delta + L_0 - 1=0$, namely, $3\delta^2+3\delta=\frac{1-L_0}{Q_0}$ and that $g$ attains its maximum in $(0,\delta_\lambda)$ at $M=\frac{1}{3}\Big(\sqrt{1+\frac{1-L_0}{Q_0}} - 1\Big)$, namely,
$$g(M)=M\Big[\frac{2}{3}(1-L_0) - Q_0M\Big]=\max\{g(\delta) : \delta\in (0,\delta_\lambda) \}.$$
Hence, if
$$\|G[P(r;\lambda)]\|_\infty + \|G[P(r;\lambda)]'\|_\infty\le M\Big[\frac{2}{3}(1-L_0) - Q_0M\Big],$$
then there exists a solution $y_\lambda$ to \eqref{poinca1} with $n=3$ satisfying \eqref{fath1} with $z$ solution to
\begin{equation}\label{eizn3}
    z = -G\big[\lambda^2r_2+ \lambda r_1 + r_0 + (2\lambda r_2+ r_1)z + r_2 z' + (3\lambda + a_2 + r_2) z^2 + 3zz' + z^3\big].
\end{equation}
In addition, if there exists a constant $0<\beta<\min\{|\alpha_1|,|\alpha_2|\}$ such that $L_{\beta}<\frac{1}{2}$ then $z$ satisfies $z,z'=O\big((I_\beta+I_{-\beta})[P(r;\lambda)]\big)$.

\medskip

On the other hand, assuming $r_i=\mu_i+\nu_i$, $i=0,1,2$ with $\mu_i\in\apr$ and $\nu_i\in BC_0$, $i=0,1,2$, direct computations lead us to obtain that
    \begin{equation}\label{eqthetan3}
        \theta=-G\big[\lambda^2\mu_2+ \lambda \mu_1 + \mu_0 + (2\lambda\mu_2+\mu_1)\theta + \mu_2\theta' + (3\lambda + a_2 +\mu_2) \theta^2 + 3\theta\theta' + \theta^3\big]
    \end{equation}
and
    \begin{equation}\label{eqpsin3}
    \begin{split}
        \psi=-G\Big[&P_\nu(\theta;\lambda)+r_\theta \psi+(r_2+\theta)\psi'+ (3\lambda+3\theta + a_2 +r_2)\psi^2 + 3\psi\psi' + \psi^3\Big],    
        \end{split}
    \end{equation}
according to equations \eqref{thetaap} and \eqref{eqpsi} in Theorem \ref{teoczero}, where
$$P_\nu(\theta;\lambda)=\lambda^2\nu_2+ \lambda \nu_1 + \nu_0 + (2\lambda\nu_2+\nu_1)\theta + \nu_2\theta' + \nu_2\theta^2$$
$$\text{and}\qquad r_\theta=2\lambda r_2+r_1+ 2(3\lambda+a_2+r_2)\theta + 3\theta^2+ 3\theta'.$$
Hence, for simplicity, we denote
$$L_{\beta,\theta}=\sum_{i=1}^2\frac{1+|\gamma_i|}{|\lambda_1-\lambda_2|}\Big[ \big\|I_{\alpha_i-{\rm sgn}(\alpha_i)\beta}[r_\theta]\big\|_\infty + \big\|I_{\alpha_i-{\rm sgn}(\alpha_i)\beta}[r_2+\theta]\big\|_\infty\Big]$$
and
$$Q_{\beta,\theta}=\sum_{i=1}^2 \frac{1+|\gamma_i|}{|\lambda_1-\lambda_2|} \bigg[ \frac{1+3|\lambda|+|a_2|}{|\alpha_i-{\rm sgn}(\alpha_i)\beta|} + \big\|I_{\alpha_i-{\rm sgn}(\alpha_i)\beta}[r_2+3\theta]\big\|_\infty\bigg].$$
Now, assume that $L_{0,\theta}<1$ and consider the function $g_\theta(\delta)=[1-L_{0,\theta}-(3\delta^2+3\delta)Q_{0,\theta}]\delta$. It is readily checked that $\delta_{\lambda,\theta}=\frac{1}{2}\Big(\sqrt{1+\frac{4(1-L_{0,\theta})}{3Q_{0,\theta}}} - 1\Big)>0$ satisfies
$3Q_{0,\theta}\delta^2 + 3Q_{0,\theta} \delta + L_{0,\theta} - 1=0$, namely, $3\delta^2+3\delta=\frac{1-L_{0,\theta}}{Q_{0,\theta}}$ and that $g_\theta$ attains its maximum in $(0,\delta_{\lambda,\theta})$ at \linebreak $M_\theta=\frac{1}{3}\Big(\sqrt{1+\frac{1-L_{0,\theta}}{Q_{0,\theta}}} - 1\Big)$, namely,
$$g_\theta(M_\theta)=M_\theta\Big[\frac{2}{3}(1-L_{0,\theta}) - Q_{0,\theta}M_\theta\Big]=\max\{g_\theta(\delta) : \delta\in (0,\delta_{\lambda,\theta}) \}.$$

 \begin{Theorem}\label{teo35n3}
 	Consider the equation \eqref{eqn3} with $r_i=\mu_i+\nu_i\in\aap(\R)$\,(resp. $\aapor$, $\ap(\R,\C,p)$ or $\ap_0(\R,\C,p)$), with $\mu_i\in\ap(\R)$ and $\nu_i\in BC_0(\R)$\,(resp. $C_{00}(\R)$, $L_0^p$ or $L^p(\R)$), $i=0,1,2$. Assume that there is a solution $\theta\in\apr$ to equation \eqref{eqthetan3}, so that $L_{0,\theta}<1$ and
  $$\|G[P_\nu(\theta;\lambda)]\|_\infty + \|G[P_\nu(\theta;\lambda)]'\|_\infty\le M_\theta\Big[\frac{2}{3}(1-L_{0,\theta}) - Q_{0,\theta}M_\theta\Big].$$
  Then, there exists a solution $y$ to equation \eqref{eqn3} satisfying \eqref{fth1n3}, where $z=\theta+\psi$ is a solution to integral equation \eqref{eizn3}, $z,z'\in\aap$ (resp. $\aapor$, $\ap(\R,\C,p)$ or $\ap_0(\R,\C,p)$) and $\psi\in BC_0^{1}(\R)$ (resp. $C_{00}^{1}(\R)$, $L_0^p$ or $L^p(\R)$) satisfies integral equation \eqref{eqpsin3}. In addition, if there exists a constant $0<\beta<\min\{|\alpha_1|,|\alpha_2|\}$ such that $L_{\beta,\theta}<\frac{1}{2}$ then $\psi$ satisfies $\psi,\psi'=O\big((I_\beta+I_{-\beta})[P_\nu(\theta;\lambda)]\big)$.
 \end{Theorem}

\section{Comments and examples}\label{exams}

Consider the equation for $n=3$
\begin{equation}\label{exn3}
    y'''+r_2(t)y''-(1-r_1(t))y'+r_0(t)y=0.
\end{equation}
We have that $\lambda_1=0$, $\lambda_2=1$ and $\lambda_3=-1$ and we shall study the case $r_2=r_1\equiv 0$. Then, we obtain that $\gamma_{11}=\alpha_{11}=1$, $\gamma_{12}=\alpha_{12}=-1$ and $\Gamma_{1i}=(-1)^{i+1}2$. Also, it is readily checked that $P(r;\lambda_i)=r_0(t)$, $L_\beta^{\lambda_i}=0$, for all $0<\beta<\min\{|\alpha_{i1}|,|\alpha_{i2} |\}=1$ and $Q_\beta^{\lambda_1}=\frac{2}{1-\beta}$, $i=1,2,3$, since $\gamma_{21}=\alpha_{21}=-1$, $\gamma_{22}=\alpha_{22}=-2$, $\gamma_{31}=\alpha_{31}=1$ and $\gamma_{32}=\alpha_{32}=2$. Hence, $M_{\lambda_1}=\frac{1}{3}\Big(\sqrt{1+\frac{1-L_0^{\lambda_1}}{Q_0^{\lambda_1}}} - 1\Big)=\frac{\sqrt{6}-2}{6}\approx 0.0749$ and
$$g_{\lambda_1}(M_{\lambda_1})=M_{\lambda_1}\Big[\frac{2}{3}(1-L_0^{\lambda_1}) - Q_0^{\lambda_1}M_{\lambda_1}\Big]=\frac{3\sqrt{6}-7}{9}\approx 0.0387.$$
If $\|G^{\lambda_1}[r_0]\|_\infty + \|G^{\lambda_1}[r_0]'\|_\infty\le \frac{3\sqrt{6}-7}{9},$
then Theorem \ref{orederzun} applies, so that there exists a solution $y_1$ to \eqref{exn3} satisfying
\begin{equation*}
\begin{split}
    y_{1}(t)= & \exp\bigg( \int_{0}^t\big[r_0(s)+3z_1(s)z_1'(s)+z^3(s)\big]\,ds\bigg)\\
    & \times\exp\bigg(- \frac{1}{2}\sum_{j=1}^{2} G_{\gamma_{1j}}\big[r_0+3z_1z'_1+z_1^3\big](t)\bigg),
\end{split}
\end{equation*}
where $z_1$ satisfies the integral equation $z = -G^{\lambda_1}\big[r_0 + 3zz' + z^3\big]$, with $G^{\lambda_1}=\frac{1}{2} G_1-\frac{1}{2} G_{-1}$. 
Furthermore, $z_1,z_1'=O([I_\beta+I_{-\beta}][r_0])$ with $0<\beta<1$.

On the other hand, $\Gamma_{2i}=(-1)^{i+1}$ and $Q_\beta^{\lambda_2}=\frac{8}{1-\beta}+\frac{12}{2-\beta}$. Hence, we choose $M_{\lambda_2}=\frac{1}{3}\Big(\sqrt{1+\frac{1-L_0^{\lambda_2}}{Q_0^{\lambda_2}}} - 1\Big)=\frac{\sqrt{105} -10}{30} \approx 0.0082$ and
$$g_{\lambda_2}(M_{\lambda_2})=M_{\lambda_2}\Big[\frac{2}{3}(1-L_0^{\lambda_2}) - Q_0^{\lambda_2}M_{\lambda_2}\Big]=\frac{21\sqrt{105}-215}{45}\approx 0.0041.$$
If $\|G^{\lambda_2}[r_0]\|_\infty + \|G^{\lambda_2}[r_0]'\|_\infty\le \frac{21\sqrt{105}-215}{45},$
then Theorem \ref{orederzun} applies, so that there exists a solution $y_2$ to \eqref{exn3} satisfying
\begin{equation*}
\begin{split}
    y_{2}(t)= &\, e^{t} \exp\bigg(-\frac{1}{2}\int_{0}^t\big[r_0(s)+3z_2^2(s)+3z_2(s)z_2'(s)+z_2^3(s)\big]\,ds\bigg)\\
    & \times\exp\bigg(- \sum_{j=1}^{2} \frac{1}{\Gamma_{2j}\gamma_{2j}} G_{\gamma_{2j}}\big[r_0+3z_2^2+3z_2z'_2+z_2^3\big](t)\bigg),
\end{split}
\end{equation*}
where $z_2$ satisfies the integral equation $z = -G^{\lambda_2}\big[r_0 +3z^2+ 3zz' + z^3\big]$,\ with $G^{\lambda_2}= G_{-1}- G_{-2}$. 
Furthermore, $z_2,z_2'=O([I_\beta+I_{-\beta}][r_0])$ with $0<\beta<1$.

For $k=3$, we obtain that $\Gamma_{3i}=(-1)^{i}$ and $Q_\beta^{\lambda_3}=\frac{8}{1-\beta} +\frac{12}{2-\beta}$. Hence, we choose $M_{\lambda_3}=\frac{1}{3}\Big(\sqrt{1+\frac{1-L_0^{\lambda_3}}{Q_0^{\lambda_3}}} - 1\Big)=\frac{\sqrt{105} -10}{30} \approx 0.0082$ and
$$g_{\lambda_2}(M_{\lambda_3})=M_{\lambda_3}\Big[\frac{2}{3}(1-L_0^{\lambda_3}) - Q_0^{\lambda_3}M_{\lambda_3}\Big]=\frac{21\sqrt{105}-215}{45}\approx 0.0041.$$
If
$\|G^{\lambda_3}[r_0]\|_\infty + \|G^{\lambda_3}[r_0]'\|_\infty\le \frac{21\sqrt{105}-215}{45},$
then Theorem \ref{orederzun} applies, so that there exists a solution $y_3$ to \eqref{exn3} satisfying
\begin{equation*}
\begin{split}
    y_{3}(t)= &\, e^{t} \exp\bigg(-\frac{1}{2}\int_{0}^t\big[r_0(s)-3z_3^2(s)+3z_3(s)z_3'(s)+z_3^3(s)\big]\,ds\bigg)\\
    & \times\exp\bigg(- \sum_{j=1}^{2} \frac{1}{\Gamma_{3j}\gamma_{3j}} G_{\gamma_{3j}}\big[r_0-3z_3^2+3z_3z'_3+z_3^3\big](t)\bigg),
\end{split}
\end{equation*}
where $z_3$ satisfies the integral equation $z = -G^{\lambda_3}\big[r_0 - 3z^2+ 3zz' + z^3\big]$, with $G^{\lambda_3}= - G_{1}+ G_{2}$. 
Furthermore, $z_3,z_3'=O([I_\beta+I_{-\beta}][r_0])$ with $0<\beta<1$.

\medskip
On the other hand, it is readily checked that the Vandermonde matrix is
$$V= \left(\begin{array}{crr}
   1 & 1 & 1 \\
  0   & 1&-1 \\
  0 & 1 & 1
\end{array}\right)
\qquad\text{and} \qquad V^{-1}=\left(\begin{array}{ccc}
   1 & 0 & -1 \\
  0   & 1/2 & 1/2 \\
  0 & -1/2 & 1/2
\end{array}\right)$$
its inverse, so that $\|V^{-1}\|_1=\|V^{-1}\|_\infty=2$. Furthermore, it is readily checked that
$$\sum_{i=2}^{3}\big[(|\lambda_k|+1)^{i-1}-|\lambda_k|^{i-1}\big]=\begin{cases}
    2& \text{if } k=1\\
    4& \text{if } k=2,3
\end{cases}$$
$$M_k\big(m(M_k)+1\big)=\begin{cases}
    \dfrac{6\sqrt{6}-8}{9}& \text{if } k=1\\[0.5cm]
    \dfrac{21\sqrt{105}-200}{225}& \text{if } k=2,3
\end{cases},$$
$$\dfrac{6\sqrt{6}-8}{9}\approx 0.093 \qquad \text{and}\qquad \dfrac{21\sqrt{105}-200}{225}\approx 0.0084 .$$
Thus, assumption \eqref{conMbase} is fulfilled and we conclude that $\{y_1,y_2,y_3\}$ are linearly independent. 

\bigskip

Now, we shall study equations \eqref{eqthetan3} and \eqref{eqpsin3} for $\lambda=\lambda_1=0$. Assume $r_0=\mu_0+\nu_0$ with $\mu_0\in\apr$ and $\nu_0\in BC_0$. For simplicity we omit the dependency of $\lambda_1$. Direct computations lead us to obtain that
    \begin{equation}\label{eqthetan3ex}
        \theta=-G\big[\mu_0 + 3\theta\theta' + \theta^3\big]
    \end{equation}
and
    \begin{equation}\label{eqpsin3ex}
        \psi=-G\Big[\nu_0 +(3\theta^2+ 3\theta')\psi +\theta\psi' + 3\theta\psi^2 + 3\psi\psi' + \psi^3\Big].
    \end{equation}
According to previous computations, if $\|G[\mu_0]\|_\infty + \|G[\mu_0]'\|_\infty\le \frac{3\sqrt 6 -7}{9}=g_{\lambda_1}(M_{\lambda_1}),$ then there exists a solution $\theta\in \ap^1$ to \eqref{eqthetan3ex}, with $\|\theta\|_\infty + \|\theta'\|_\infty\le \frac{\sqrt{6} -2}{6}=M_{\lambda_1}$ and \linebreak $\theta=O([I_\beta+I_{-\beta}][\mu_0])$ with $0<\beta<1$.
Hence, $P_\nu(\theta;\lambda)=  \nu_0$, $r_\theta= 3\theta^2+ 3\theta'$,
$$L_{\beta,\theta}= \big\|I_{1-\beta}[3\theta^2+ 3\theta']\big\|_\infty + \big\|I_{1-\beta}[\theta]\big\|_\infty+ \big\|I_{-1+\beta}[3\theta^2+ 3\theta']\big\|_\infty + \big\|I_{-1+\beta}[\theta]\big\|_\infty$$
and
$$Q_{\beta,\theta}= \frac{1}{|1-\beta|} + \big\|I_{1-\beta}[3\theta]\big\|_\infty +\frac{1}{|-1+\beta|} + \big\|I_{-1+\beta}[3\theta]\big\|_\infty.$$
Notice that according to Remark \ref{cotadervG} we deduce that
$$L_{0,\theta}\le 6\|\theta\|_\infty^2+6\|\theta'\|_\infty+2\|\theta\|_\infty\le 6M^2_{\lambda_1}+6M_{\lambda_1}=\frac{\sqrt{6}-1}{3}<\frac{1}{2}<1$$
so that, $M_\theta=\frac{1}{3}\Big(\sqrt{1+\frac{1-L_{0,\theta}}{Q_{0,\theta}}} - 1\Big)$ is well defined and
$$g_\theta(M_\theta)=M_\theta\Big[\frac{2}{3}(1-L_{0,\theta}) - Q_{0,\theta}M_\theta\Big]=\max\{g_\theta(\delta) : \delta\in (0,\delta_{\lambda,\theta}) \},$$
where $\delta_{\lambda,\theta}=\frac{1}{2}\Big(\sqrt{1+\frac{4(1-L_{0,\theta})}{3Q_{0,\theta}}} - 1\Big)>0$ and it satisfies $3\delta^2+3\delta=\frac{1-L_{0,\theta}}{Q_{0,\theta}}$
Furthermore, we get that
$$2\le Q_{0,\theta}=2+\|I_1[3\theta]\|_\infty+\|I_{-1} [3\theta]\|_\infty\le 2+6\|\theta\|_\infty\le \sqrt 6$$
and $1-L_{0,\theta}\ge \frac{4-\sqrt{6}}{3}$.  This implies that
$$\frac{1}{2}\ge \frac{1}{Q_{0,\theta}}\ge \frac{1-L_{0,\theta}}{Q_{0,\theta}}\ge \frac{\frac{4-\sqrt{6}}{3}}{\sqrt{6}} = \frac{2\sqrt{6} - 3}{9}$$
and
$$\frac{\sqrt{6}-2}{6}=\frac{1}{3}\bigg(\sqrt{1+ \frac{1}{2}} - 1\bigg)\ge M_\theta\ge \frac{1}{3}\bigg(\sqrt{1+ \frac{2\sqrt{6} - 3}{9}} - 1\bigg)=\frac{\sqrt{6+2\sqrt{6}}}{9} - \frac{1}{3}.$$
Taking into account previous inequalities we obtain that
$$\frac{2}{9}\Big(\sqrt{6+2\sqrt{6}}-3\Big)\le M_\theta Q_{0,\theta}\le \sqrt{6}\cdot \frac{\sqrt{6}-2}{6}$$
and
$$g_\theta(M_\theta)\ge \frac{\sqrt{6+2\sqrt{6}}}{9} - \frac{1}{3}\bigg[\frac{2}{3}\cdot \frac{4-\sqrt{6}}{3} - \frac{6-2\sqrt{6}}{6} \bigg]=\frac{\sqrt{6}-1}{81}\Big(\sqrt{6+2\sqrt{6}}-3\Big).$$
If $\|G[\nu_0]\|_\infty + \|G[\nu_0]'\|_\infty\le \frac{\sqrt{6}-1}{81}\Big(\sqrt{6+2\sqrt{6}}-3\Big),$ then Theorem \ref{teo35n3} applies, namely, there exists a solution $\psi\in BC_0$ to \eqref{eqpsin3ex}, with $\|\psi\|_\infty+\|\psi'\|_\infty\le M_\theta$ and $\psi=O([I_\beta+I_{-\beta}][\nu_0])$ with $0<\beta<1$. Thus, there exists a solution $y$ to equation \eqref{poinca1} satisfying \eqref{fath1}, where $z=\theta+\psi$.

\bigskip
Suppose that $r_0=\mu_0+\nu_0$ with $\mu_0(t)=\eta_1\left[2+\cos t +
\cos(\sqrt{2} t)\right]$ and $\nu_0(t)={\eta_2\over 1+t^2}$, where
$\eta_1,\eta_2\ge0$. It is clear that $\mu_0$ and $\nu_0$ are
non-negative functions, $\mu_0\in\aprc$, $\nu_0\in
C_{00}(\R,\C)\cap L^1(\R)$. 
Direct computations show that
$$G^{\lambda_1}[\mu_0](t)=\frac{1}{2}G_1[\mu_0](t)-\frac{1}{2}G_{-1}[\mu_0](t)=--2\eta_1-\frac{\eta_1}{2}\cos t -\frac{\eta_1}{3}\cos(\sqrt{2} t),$$
since for any $\lambda\in\R$ and integration by parts
$$\int_t^{\mp\infty}e^{\mp(t-s)}\cos(\lambda s)\,ds=\mp \cos(\lambda t) \pm \lambda \int_t^{\mp\infty}e^{\mp (t-s)}\sin(\lambda s)\,ds
={\mp \cos(\lambda t)-\lambda\sin(\lambda t)\over 1+\lambda^2}. $$
Furthermore, we get that
$$G^{\lambda_1}[\mu_0]'(t)=\frac{\eta_1}{2}\sin t +\frac{\eta_1}{3}\sqrt{2}\sin(\sqrt{2} t),$$
and
$$\|G^{\lambda_1}[\mu_0]\|_\infty+\|G^{\lambda_1}[\mu_0]'\|_\infty\le \frac{17}{6}\eta_1 + \eta_1\Big(\frac{1}{2}+\frac{\sqrt{2}}{3}\Big)\le \frac{3\sqrt{6}-7}{9}$$
Theorem \ref{orederzun} applies. $\|G_{\pm 1}[\nu_0]\|_\infty\le \|\nu_0\|_\infty$ and $\|\nu_0\|_\infty=\eta_2$, so that
$\|G^{\lambda_1}[\nu_0]\|_\infty\le \|\nu_0\|_\infty$ and $\|G^{\lambda_1}[\nu_0]'\|_\infty\le \|\nu_0\|_\infty$. Hence, it is enough
$$2\eta_2\le \frac{\sqrt{6}-1}{81}\Big(\sqrt{6+2\sqrt{6}}-3\Big)$$
to apply Theorem \ref{teo35n3}. Similar computations can be perform for $\lambda_2$ and $\lambda_3$.

\medskip
\begin{center}
{\bf Acknowledgements}    
\end{center}
The second author has been partially supported by grant Fondecyt Regular N$^\circ$ 1201884. The third author has been supported by grant Fondecyt Regular N$^\circ$ 1170466.

\end{document}